\documentclass[12pt,a4paper,reqno]{amsart}

\usepackage{amssymb}
\usepackage{amscd}
\usepackage{amsfonts}
\usepackage{setspace}
\usepackage{version}
\usepackage{color}

\numberwithin{equation}{section}
\theoremstyle{plain}
  \newtheorem{theorem}[subsection]{Theorem}
  \newtheorem*{mainthm}{Main Theorem}
  \newtheorem{proposition}[subsection]{Proposition}
  \newtheorem*{proposition_w-reduction2}{Proposition \ref{w-reduction2}}
  \newtheorem{lemma}[subsection]{Lemma}
  \newtheorem{corollary}[subsection]{Corollary}
  \newtheorem{claim}{Claim}
  \newtheorem*{Claim}{Claim}
  
  \newtheorem{definition}[subsection]{Definition}

\theoremstyle{remark}
  \newtheorem*{remarks}{Remarks}
  \newtheorem*{remark}{Remark}

\renewcommand{\leq}{\leqslant}
\renewcommand{\geq}{\geqslant}
\newcommand{\Mod}[1]{\;(\mathrm{mod}\,#1)}

\newcommand{\stsubsection}[1]{     \subsection*[#1]{\sc #1}}
\newcommand\E{\mathbb{E}}
\newcommand\Z{\mathbb{Z}}
\newcommand\R{\mathbb{R}}

\newcommand\C{\mathbb{C}}

\newcommand\N{\mathbb{N}}
\newcommand\W{\overline W}
\newcommand\eps{\varepsilon}
\newcommand\vol{\operatorname{vol}}
\newcommand\lcm{\operatorname{lcm}}

\newcommand{\dsum}{\sideset{}{^{\prime}}{\sum}}

\begin{document}
\title{Correlations of the divisor function}

\author{Lilian Matthiesen}
\address{Centre for Mathematical Sciences\\
Wilberforce Road\\
Cambridge CB3 0WA\\
UK }
\email{L.Matthiesen@dpmms.cam.ac.uk}

%\keywords{keywords}

\begin{abstract}
Let $\tau(n)= \sum_{d} 1_{d|n}$ denote the divisor function.
Based on Erd\H{o}s's fundamental work on sums of multiplicative functions
evaluated over polynomials, we construct a pseudorandom majorant for 
a slightly smoothed version of $\tau$.
By means of the nilpotent Hardy-Littlewood method we give an asymptotic
for the following correlation
$$\E_{n \in [-N,N]^d \cap K} \prod_{i=1}^t \tau(\psi_i(n))~,$$
where $\Psi=(\psi_1, \dots, \psi_t)$ is a non-degenerated system of
affine-linear forms no two of which are affinely related, and where $K$
is a convex body.
\end{abstract}

\maketitle
\tableofcontents

\addtocontents{toc}{\protect\setcounter{tocdepth}{1}}

\section{Introduction}\label{sec1}

Questions concerning the distribution of the values of elementary
arithmetic functions play a central role in analytic number theory. We
mention two classes of such questions, both of which are related to the
results of this paper.

The first class concerns asymptotics for sums 
$$\sum_{M \leq  n \leq N+M} f(|P(n)|)$$  
of multiplicative functions evaluated over polynomials, a direction which
has been substantially influenced by Erd\H{o}s's work on the sum 
$\sum \tau(P(n))$, see \cite{erdos}. 
We shall employ some ideas introduced in that paper.
For newer work on this type of question, see for instance
\cite{nair-tenenbaum} and the references therein.

A second class considers linear correlations. 
Write $[N]$ to denote the set of numbers $\{1,\dots, N\}$,
let $f : [N] \to \R$ be an arithmetic function and let
$\psi_1,\dots,\psi_t : \Z^d \to \Z$ be affine-linear forms. 
Then we ask for an asymptotic to the correlation
\begin{equation}\label{star}
\sum_{n \in K \cap \Z^d} f(\psi_1(n)) \dots f(\psi_t(n)),
\end{equation}
where $K \subseteq [-N,N]^d$ is a convex body such that 
$\psi_i(K) \subseteq [1,N]$ for each $i \in [t]$. 
Questions of this second type include the generalised Hardy-Littlewood
conjecture, which predicts, based on a probabilistic model for the prime
numbers, an asymptotic for \eqref{star} when $f=\Lambda$ is given by the
von Mangoldt function. 
Note that the frequency of arithmetic progressions of a fixed length $t$
in the set of primes can be expressed as a special case 
$$\E\big( \Lambda(n_1)\Lambda(n_1 + n_2) \dots \Lambda(n_1 + (t-1)n_2)
\mid n_1 + (t-1)n_2 \leq N \big)
$$
of the $f=\Lambda$ instance of \eqref{star}. 
The generalised Hardy-Littlewood conjecture has been resolved in the
series of papers
\cite{green-tao-linearprimes, green-tao-nilmobius,
green-tao-polynomialorbits, gtz} for those cases where no two forms
$\psi_i$ and $\psi_j$ are affinely related. 
(Thus the prime $k$-tuples conjecture, which concerns the asymptotic
behaviour of $$\E_{n \leq N} \Lambda(n + h_1 )\Lambda(n + h_2) \dots
  \Lambda(n + h_k)$$
for any $k$-tuple of integers $h_1, \dots, h_k$, remains unsettled.)

The general approach that was used in the aforementioned partial
resolution of the Hardy-Littlewood conjecture is described as the
`nilpotent Hardy-Littlewood method' in \cite{green-tao-nilmobius}.
This method can be employed to resolve questions of the above second kind,
provided the function $f$ involved shows a certain amount of random-like
behaviour. 
It resembles the classical method in that this approach too
requires a (suitably adapted) major and a minor arc analysis (Section
\ref{non-correlation-section}), cf.~\cite[\S4]{green-tao-nilmobius} for a
discussion of this analogy.
A very central role in this method is assigned to \emph{pseudorandom
majorant} functions.
We shall explain the reason for this and its role at the start of
Section \ref{majorant-section}. 
In the case of the divisor function, the construction of the majorant
constitutes the principal task that needs to be accomplished in
order to apply the method and thus in order to obtain an asymptotic for
\eqref{star} with $f=\tau$.

For an application of the nilpotent Hardy-Littlewood method the
function $f$ is required to have asymptotic density, that is, to satisfy
$$\E_{n\leq N} f(n) = \delta+o(1)$$
for some absolute constant $\delta \geq 0$.
For this reason, we shall work not with the divisor function itself, but
with the normalised divisor function
${\tilde \tau}  : [N] \rightarrow \R_{\geq 0}$ which is defined by
\[ {\tilde \tau} (n) := \frac{1}{\log N}\sum_{d | n} 1\] 
and has asymptotic density $\delta = 1$.

A \emph{pseudorandom majorant} for $f$ is a function 
$\nu: [N] \to \R_{\geq 0}$
such that $|f(n)| \leq C\nu(n)$ pointwise (for some absolute constant
$C$), and which resembles a random measure in the following sense. 
The total mass of $\nu$ is approximately $1$, that is
$\E_{n \leq N} \nu(n) = 1 + o(1)$, and two further conditions modelling
independence are satisfied. These are the linear forms and correlation
conditions from \cite{green-tao-linearprimes}. 
The linear forms condition requires asymptotics of the form
$$\E_{n \in K \cap \Z^d} \nu(\psi_1(n)) \dots \nu(\psi_t(n))=1+o(1).$$
Note that this is \eqref{star} for the majorant $\nu$ instead of $f$.
Thus, to enable us to check this condition, the pseudorandom majorant
$\nu$ has to be of a form that allows a good understanding of its value
distribution.
In particular, assuming that one failed to establish \eqref{star} for $f$
directly and hence resorted to other methods of approach, the majorant
has to be sufficiently easier to understand than the function $f$.

In the course of the above cited work on \eqref{star} for the von Mangoldt
function, the problem of finding an asymptotic for \eqref{star} was also
addressed for $f = \mu$, the M\"obius function.
A key feature of both functions $\mu$ and $\Lambda$ is that they
show some regularity in their growth.
$\mu$ is bounded by $1$ pointwise, whereas $\Lambda$ grows not 
faster than $\log$. 
This regularity is of advantage for the task of constructing a function
that is simple enough that one can check the linear forms condition, and
which simultaneously satisfies the majorant and the density condition. 

The divisor function $\tau(n) = \sum_{d|n}1$, on the other hand, is known
for its irregularities in distribution.
The moments $\E_{n \leq N} \tau (n)^p \sim (\log N)^{2^p-1}$ grow rapidly
in $p$.
While $\tau$ has an `approximate' normal order, that is for every
$\eps > 0$ all but $o(N)$ positive integers $n \leq N$ satisfy
$(\log N)^{(1-\eps)\log 2} < \tau(n) < (\log N)^{(1+\eps)\log 2}$, a
theorem of Birch \cite{birch} implies that it does not have a normal order
in the sense of Hardy and Ramanujan.
Instead there is a gap between the `approximate' normal order 
$(\log N)^{\log 2}$ and the average order $\log N$, which results from few
exceptionally large values of $\tau$.
In particular, $\tau(n)$, and similarly ${\tilde \tau} (n)$, can be as
large as $\exp(c \log n / \log \log n)$, see
\cite[\S18.1 and \S22.13]{hardy-wright}.

We shall show that, nonetheless, there is a pseudorandom majorant 
$\nu: [N] \to \R_{\geq 0}$ for (a $W$-tricked version of) $\tilde \tau$,
and that the same basic method that was employed to deal with $f=\mu$ and
$f=\Lambda$ can also be employed in this case:
The existence of this majorant in combination with the recent complete
resolution of the Inverse Conjectures for the Gowers norms \cite{gtz}
allows us to deduce an asymptotic for 
$\sum_{ n \in K \cap \Z^d} {\tilde \tau} (\psi_1(n)) \dots {\tilde \tau}
(\psi_t(n))$
under the already mentioned condition that no two forms $\psi_i$
and $\psi_j$ are affinely related.
 
\stsubsection{Notation and statement of the main result} We recall some
notation from \cite{green-tao-linearprimes} in order to state the result
precisely.

\begin{definition}[Affine-linear forms] Let $d,t \geq 1$ be
integers. An \emph{affine-linear form} on $\Z^d$ is a function $\psi:
\Z^d \to \Z$ which is the sum $\psi= \dot{\psi} + \psi(0)$ of a linear
form $\dot{\psi}: \Z^d \to \Z$ and a constant $\psi(0) \in \Z$. A
\emph{system of affine-linear forms} on $\Z^d$ is a collection $\Psi =
(\psi_1, \dots, \psi_t)$ of affine-linear forms on $\Z^d$ that is
required to satisfy the following non-degeneracy condition: 
no affine-linear form is constant, 
no restriction of $\Psi$ to a single variable is constant and 
no two forms are rational multiples of each other.
\end{definition}

Let $N$ be a (usually large) positive integer, and let $L$ be a fixed
positive integer. 
Throughout this paper we will assume that the coefficients of the linear
part $\dot \Psi$ of the affine-linear system we work with are bounded by
$L$. 
The constant term $\Psi(0)$ may depend on the cut-off $N$, but
we will require that the convex set $K$ is such that 
$K \subseteq [-N,N]^d$ and $\Psi(K) \subset [1,N]^t$.
We furthermore assume that no $\dot \psi_i,\dot \psi_j$, $i\not=j$, are
linearly dependent.

As we will show, the asymptotic behaviour of
\begin{equation}\label{introd;eq1}
 \sum_{ n \in K \cap \Z^d} {\tilde \tau} (\psi_1(n)) \dots {\tilde \tau}
  (\psi_t(n))
 = \frac{1}{(\log N)^t} 
  \sum_{d_1, \dots, d_t}
  \sum_{ n \in K \cap \Z^d}
  1_{d_1|\psi_1(n)} \dots 
  1_{d_t|\psi_t(n)}
\end{equation}
is then determined by the local behaviour of the affine-linear system
modulo small primes. 
To make this precise, we proceed to define local factors at primes.

For a given system $(\psi_1,\dots,\psi_t)$ of affine-linear forms,
positive integers $d_1, \dots, d_t$ and their least common multiple 
$m:= \lcm(d_1, \dots, d_t)$ define \emph{local divisor densities} by 
$$\alpha(d_1, \dots, d_t) := \E_{n \in (\Z/m\Z)^d} \prod_{i \in [t]}
1_{\psi_i(n) \equiv 0 \Mod{d_i}}~.$$
The Chinese remainder theorem implies that $\alpha$ is multiplicative.
Thus, we restrict attention to what happens at prime powers $d_i=p^{a_i}$
for a fixed prime $p$. 
If the forms $\psi_i$ were independent, one would expect 
$\alpha(p^{a_1},\dots,p^{a_t}) = p^{-a_1} \dots p^{-a_t}$. 
The prime powers of $p$ would then contribute to \eqref{introd;eq1} a
factor of
$$\sum_{a_1, \dots, a_t} p^{-a_1} \dots p^{-a_t} = (1 - p^{-1})^{-t}~.$$
We therefore introduce for each prime $p$ a \emph{local factor}
$$\beta_p := (1-p^{-1})^t \sum_{a_1, \dots, a_t \in \N}
  \alpha(p^{a_1},\dots,p^{a_t})$$
which measures the irregularities of the divisor densities of the given
system $\Psi$ of affine-linear forms.
As will be checked in the next section, the local factors satisfy the
estimate
$\beta_p = 1 + O_{t,d,L}(p^{-2})$. 
Thus, in particular, their product $\prod_p \beta_p$ converges.

Our main result is the following local-global principle.

\begin{mainthm} Let $N,d,t,L$ be positive integers and let $\Psi =
(\psi_1,\dots,\psi_t) : \Z^d \to \Z^t$ be a system of
affine-linear forms whose coefficients of non-constant terms are bounded
by $L$ and for which any $\dot \psi_i$, $\dot \psi_j$, $i \not= j$, are
linearly independent. 
Then
\[ \sum_{n \in K \cap \Z^d} \prod_{i=1}^t {\tilde \tau}
(\psi_i(n))
   = \vol(K) \prod_p \beta_p + o_{t,d,L}(N^d) 
\]
for any convex body $K \subseteq [-N,N]^d$ such that
$\Psi(K) \subset [1,N]^t$.
\end{mainthm}
Observe that this result only gives suitable information when 
$N^d \ll \vol(K)$.

The corresponding asymptotic for the divisor function is an
immediate consequence:

\begin{corollary}[Correlations of the divisor function] With the
assumptions of the Main Theorem, the divisor function $\tau$ satisfies
 $$ \sum_{n \in K \cap \Z^d} 1_K(n) \prod_{i=1}^t \tau(\psi_i(n))
   = (\log N)^t \vol(K) \prod_p \beta_p + o_{t,d,L}(N^d \log^t N)~.$$
\end{corollary}

The special case of $d=2$ and $t=3$ of this corollary also appears as
special case of \cite[Thm 3]{browning} when setting all the $d_i$ and
$D_i$ to equal $1$ in the statement of the latter.
In contrast to our result, \cite[Thm 3]{browning} gives, with a saving of
a power in $N$, a good explicit error term.

The condition that no two forms $\dot \psi_i$ and $\dot \psi_j$ are
linearly dependent, which the main theorem places upon the affine-linear
system $\Psi$, is equivalent to saying that the affine-linear system
$\Psi$ has \emph{finite complexity}, a notion introduced in
\cite{green-tao-linearprimes}.
The \emph{infinite complexity} case includes problems of just
one free parameter, like the one of estimating 
\begin{equation}\label{inf complexity example}
 \E_{n \leq N} \tau (n + a_1) \dots \tau (n + a_k)~.
\end{equation}
These remain untouched, as they cannot be addressed by the nilpotent
Hardy-Littlewood method. To place the task of estimating 
\eqref{inf complexity example} into context, we mention that
Ingham \cite{Ing27JLMS} proves the asymptotic
$$ \sum_{n=1}^N \tau(n) \tau(n+a) 
   = \frac{6}{\pi^2} \sigma_{-1} (a)~ N \log^2 N + O(N \log N)~,$$
where $\sigma_{-1} (a) = \sum_{d|a} d^{-1}$. 
No asymptotics are known when $k \geq 3$;
c.f.~\cite[Thm 2]{browning} for a recent result into the direction of
gaining asymptotics in the $k=3$ case.

A subsequent paper \cite{m-diagquadratics} considers the problem of
the type \eqref{star} for other arithmetic functions such as $r(n)$, the
number of representations of $n$ as a sum of two squares.
This has some natural arithmetic consequences concerning the
number of simultaneous integer zeros of pairs of certain diagonal
quadratic forms, which are, in the 8-variables case, out of reach
of the classical Hardy-Littlewood method as it currently stands.

\section{Local divisor densities}

This section contains some lemmas involving local divisor densities that
are repeatedly used in analysing singular products. We also provide an
estimate for $\beta_p$.

Let $\Psi=(\psi_1,\dots,\psi_t):\Z^d \to \Z^t$ be a system of
affine-linear forms whose linear coefficients are bounded by $L$, let 
$K \subseteq [-N,N]^d$ be a convex body, and let $d_1, \dots, d_t$ be
integers. 
Divisibility events of the form
\[\sum_{n\in \Z^d \cap K} \prod_{i\in [t]} 1_{d_i|\psi_i(n)}\]
will naturally occur quite frequently in this paper.
As in \cite{green-tao-linearprimes}, the main tool to deal with these
divisibility events is a simple volume packing lemma.
\begin{lemma}[Volume packing argument]
\label{volume-packing} 
Let $B$ be a positive integer, let $K \subset \R^d$ be a convex body that
is contained in some translate of $[-B,B]^d$ and let $\Psi:\Z^d \to \Z^t$
be a system of affine-linear forms. Then 
\[
 \sum_{n\in \Z^d \cap K} \prod_{i\in [t]} 1_{d_i|\psi_i(n)} 
 = \vol(K) \alpha (d_1, \dots, d_t) + O(B^{d-1} \lcm(d_1,\dots, d_t))~.
\]
\end{lemma}
\begin{proof}
Let $\delta = \lcm(d_1,\dots, d_t)$ and cover $K$ by translates
$\delta \Z^d + [0,\delta)^d $ of the box $[0,\delta)^d$. Each box
contains $\delta^d \alpha (d_1, \dots, d_t)$ points $n$ such that
$\prod_{i\in [t]} 1_{d_i|\psi_i(n)} = 1$. Any box that does not lie
completely inside $K$ is contained in the $2 \delta$-neighbourhood of the
boundary of $K$, which has by
\cite[Corollary A.2]{green-tao-linearprimes} a volume of order
$O_d(\delta B^{d-1})$. Putting things together yields the result. 
\end{proof}

We proceed to analyse the multiplicative function $\alpha = \alpha_{\Psi}$
more closely.
If $p$ is large compared to $t$, $d$, $L$, then
\begin{equation}\label{alpha-bound-0}
 \alpha(p^{a_1}, \dots, p^{a_t}) = p^{-a_j}~,
\end{equation}
when $a_j$ is the only non-zero exponent. 
A prime $p$ is called \emph{exceptional} (with respect to $\Psi$) when
there are forms $\psi_i$, $\psi_j$ in the system that are affinely related
modulo $p$.
If $a_i,a_j>0$, then considering the number of solutions 
$n\in(\Z/p^{\max(a_i,a_j)}\Z)^d$ to
$\psi_i(n) \equiv 0 \Mod{p^{a_i}}$, $\psi_j(n)\equiv 0 \Mod{p^{a_j}}$
yields $\alpha(p^{a_1}, \dots, p^{a_t}) \leq p^{-a_i - a_j}$ if $\psi_i$
and $\psi_j$ are not affinely related. 
Thus, if $p$ is not an exceptional prime, one has, 
with $a_{\max}:= \max_i a_i$,
\begin{equation}\label{alpha-bound}
\alpha(p^{a_1}, \dots, p^{a_t}) \leq p^{-a_{\max} - 1}~,
\end{equation}
if there are at least two non-zero exponents.

\begin{lemma}[Contribution from dependent divisibility events]
\label{dependent_div_events}
Let $\Psi$ be as above and let $p$ be an unexceptional prime. Then
$$
 \sum_{\substack{a_1, \dots, a_t \geq 0 \\ 
   \text{at least two $a_i \not= 0 $} }} \alpha(p^{a_1}, \dots, p^{a_t})
 \ll_{t,d,L} \frac{1}{p^{2}}~.
$$
\end{lemma}
\begin{proof}
The number of $t$-tuples $(a_1, \dots, a_t)$ of non-negative integers with
$\max_i a_i = j$ is at most $tj^{t-1}$. This together with the bound
\eqref{alpha-bound} yields
$$
 \sum_{\substack{a_1, \dots, a_t \geq 0 \\ 
   \text{at least two $a_i \not= 0 $} }} \alpha(p^{a_1}, \dots, p^{a_t})
 \ll_{L,t,d} \sum_{j\geq 1} \frac{j^t}{p^{j+1}} 
 = \sum_{k\geq 2} \frac{1}{p^{k}} 
   \left( \frac{(2k-1)^t}{p^{k-2}} + \frac{(2k)^t}{p^{k-1}} \right)~.
$$ 
There is $p_0$ such that whenever $p>p_0$ then all the brackets in
the last sum are less than $1$, except the bracket for $k=2$.
Thus, for $p>p_0$
$$
 \sum_{\substack{a_1, \dots, a_t \geq 0 \\ 
   \text{at least two $a_i \not= 0 $} }} \alpha(p^{a_1}, \dots, p^{a_t})
 \ll_{L,t,d} \frac{1}{p^2} + \sum_{k\geq 3}\frac{1}{p^k} 
 \ll_{L,t,d} \frac{1}{p^2}~.
$$
\end{proof}

The following lemma immediately implies the convergence of 
$\prod_p \beta_p$ whenever $\Psi$ contains no two forms $\psi_i$ and
$\psi_j$ that are affinely dependent, and thus every exceptional prime is
bounded by $O_{t,d,L}(1)$. 
\begin{lemma}
Let $\Psi$ be as above and let $p$ be an unexceptional prime, then
\begin{equation}\label{beta_p-bound}
\beta_p = 1 + O_{t,d,L}(p^{-2})~.
\end{equation}
\end{lemma}
\begin{proof}
By Lemma \ref{dependent_div_events} and the bound \eqref{alpha-bound-0}
\begin{align*}
 \beta_p 
&= (1-p^{-1})^t \sum_{a_1, \dots, a_t \in \N}
   \alpha(p^{a_1},\dots,p^{a_t}) \\
&= \left(1 - \frac{t}{p} + O_{t,d,L}(p^{-2})\right) 
   \left(1 + \frac{t}{p} + O_{t,d,L}(p^{-2})\right) \\
&= 1 + O_{t,d,L}(p^{-2})
\end{align*}
which proves the result.
\end{proof}

\section{Some arithmetical lemmas and a reduction}
\label{section: arithmetical lemmas}

In this section we record for later reference some early lemmas from
\cite{erdos}, adapted to our purposes, and deduce a reduction of the Main
Theorem.

\begin{lemma}[$k$-th moment bound for the divisor function $\tau$]
\label{divisor-bound}
Let $k$ be an integer and let
$\Psi=(\psi_1, \dots, \psi_t): \Z^m \to \Z^t$ be a system of
affine-linear forms whose linear coefficients are bounded by $L$.
Suppose that $K \subseteq [-N,N]^d$ is a convex set such that
$\Psi(K)\subseteq [1,N]^t$.
Then
$$\E_{n \in \Z^m \cap K} \prod_{i \in [t]} \tau^k(\psi_i(n)) 
\ll_{t,m,L} (\log N)^{O_{k,t}(1)}~.$$
\end{lemma}
\begin{proof}
Recall that we assumed that no affine-linear form is constant.
H{\"o}lder's inequality and the estimate 
$\E_{n \in \Z^m \cap K} 1_{d|\psi_i(n)} \ll_{m,L} \frac{1}{d}$ 
imply the lemma as follows.
\begin{align*}
  \E_{n \in \Z^m \cap K} \prod_{i \in [t]} \tau^k(\psi_i(n)) 
 &\leq \prod_{i \in [t]} 
   \bigg(\E_{n \in \Z^m \cap K} 
    \tau^{kt}(\psi_i(n))\bigg)^{1/t}\\
 &\leq \prod_{i=1}^t \left(
    \sum_{\substack{d_1,\dots, d_{tk} \leq N }}
    \E_{n \in \Z^m \cap K} \prod_{j = 1}^{tk} 1_{d_j|\psi_i(n)} 
    \right)^{1/t}\\
 &\ll_{m,L} \sum_{\substack{d_1,\dots, d_{tk} \leq N }}
    \frac{1}{\lcm(d_1, \dots, d_{tk})} ~.
\end{align*}
If a fixed positive integer $d$ is the least common multiple of $\ell$
numbers $d_{1},\dots,d_{\ell}$, then to each choice of $(d_j)_{j\in
[\ell]}$ there corresponds a unique factorisation of $d$ into
$(2^{\ell}-1)$ factors $d^{(I)} \geq 1$, one for each non-empty subset $I$
of indices, which are defined by the property that $p^a\|d^{(I)}$ implies
that there is $b\geq 0$ such that 
$p^{b}\|\lcm(d_j : j \not\in I)$ and $p^{b+a}\|\gcd(d_j : j \in I)$.
By this factorisation, the last expression in the chain above is
seen to be bounded by
\begin{align*}
 &\leq \prod_{j=1}^{2^{tk}-1}\sum_{e_j \leq N }
    \frac{1}{e_j} \ll (\log N)^{2^{tk}-1}~.
\end{align*}

\end{proof}

\begin{lemma}[``rough'' numbers are rare, \cite{erdos}]\label{(i)-bound}
Suppose $\Psi=(\psi_1,\dots,\psi_t):\Z^d \to \Z^t$ is affine-linear and
its linear coefficients are bounded by $L$.
Let $K \subseteq [-N,N]^d$ be a convex body such that
$\Psi(K)\subset[1,N]^t$.
Let $C_1 > 1$ be a parameter and let $S_1$ be the set of $m \in \Z$
which are divisible by a large
proper prime power $p^a > \log^{C_1} N$, $a\geq2$.
Then the density of $n \in \Z^d \cap K$
such that $\psi_i(n) \in S_1$ for at least one $i \in [t]$ is bounded by 
$$\sum_{i \in [t]} \E_{n \in \Z^d \cap K} 1_{\psi_i(n) \in S_1}
\ll_{L,d,t} \log^{-C_1/2} N~.$$
\end{lemma}
\begin{proof} 
This is a straightforward adaption of the one-dimensional estimate.
Note that $\E_{n \in \Z^d \cap K}
1_{p^a|\psi_i(n)} \ll_{L,d}
p^{-a}$ for all primes $p$. Let $a(p)$ be the smallest exponent $a\geq2$
for which $p^a > \log^{C_1} N$. We then have
\begin{align*}
 \sum_{i \in [t]} \E_{n \in \Z^d \cap K} 1_{\psi_i(n) \in S_1}
&\leq \sum_{p} \E_{n \in \Z^d
 \cap K} \sum_{i\in[t]} 1_{p^{a(p)}|\psi_i(n)}\\
 &\ll_{L,d} \sum_{p \leq \log^{C_1/2} N} t \log^{-C_1} N + 
     \sum_{p > \log^{C_1/2} N} t p^{-2}\\
 &\ll_{L,d,t} \log^{-C_1/2} N~.
\end{align*}
\end{proof}

\begin{lemma}[``smooth'' numbers are rare, \cite{erdos}]\label{(ii)-bound}
Let $\Psi$ and $K$ be as in the previous lemma, let $0<\gamma < 1$ be a
parameter
and let $S_2$ be the set of smooth $m \in \N$, that is, $m$ for which
\begin{equation}\label{product}
\prod_{\substack{p^a\|m \\ p \leq N^{1/(\log \log N)^3}}} p^a
\geq N^{\gamma/\log \log N}~. 
\end{equation}
Then the density of $n \in \Z^d \cap K$ for which $\psi_i(n) \in S_2$ for
at least one $i \in [t]$ is bounded by
$$ \sum_{i\in[t]}
\E_{n \in \Z^d \cap K} 1_{\psi_i(n) \in S_2}
\ll_{L,d,t,\gamma, C_1} \log^{-C_1/2} N~,$$
where $C_1$ is as in Lemma \ref{(i)-bound}.
\end{lemma}
\begin{proof} 
Suppose that $\psi_i(n) \in S_2$ but does not belong to the set
$S_1$ from the previous lemma at the same time. 
Then each prime power in the product \eqref{product} for $m=\psi_i(n)$ is
in particular $\ll_{C_1} N^{1/(\log \log N)^3}$. 
Since $\psi_i(n)>N^{\gamma/\log \log N}$, we then have
$$\tau(\psi_i(n)) \geq 2^{\omega(\psi_i(n))} \gg_{C_1} 2^{(\log \log N)^3
\gamma/ \log \log N} \gg_{\gamma, C_2} (\log N)^{C_2}$$
for any positive constant $C_2$. 
For each value of $i$, the bound 
$\E_{n \in \Z^m \cap K} \tau(\psi_i(n)) 
= \sum_{d \leq N} \E_{n \in \Z^m \cap K} 1_{d|\psi_i(n)}
\ll_{L,m} \log N$ implies that this can happen only on a set of
density not exceeding $O((\log N)^{1-C_2})$. 
The result follows with $C_2\geq 1 + C_1/2$.
\end{proof}

The next lemma shows that $S_1$ and $S_2$ are exceptional sets
for the divisor function.
\begin{lemma}[Contribution from the exceptional sets $S_1$ and $S_2$]
\label{paucity} 
Let $\Psi$ be as before and let $C_3\geq1$ be a parameter. 
For sufficiently large $C_1$, we have
$$
 \sum_{i\in[t]}
 \E_{n \in \Z^d \cap K} 
 1_{\psi_i(n) \in S_1 \cup S_2} 
 \prod_{j \in [t]} \tilde\tau(\psi_j(n)) 
 \ll_{t,d,L} (\log N)^{-C_3}~.
$$
\end{lemma}
\begin{proof}
 This follows by the Cauchy-Schwarz inequality from lemmata
\ref{divisor-bound}, \ref{(i)-bound} and \ref{(ii)-bound} provided $C_1$
is chosen large enough. 
\end{proof}

The previous lemma reduces the task of proving the Main Theorem as
follows. 

\begin{proposition} \label{tau-bar-reduction}
Let $\bar \tau : \Z \to \R$ be any function that agrees with 
$\tilde\tau$ on the complement of $S_1 \cup S_2$ and satisfies 
$0 \leq \bar \tau (n) \leq \tilde \tau (n)$ for 
$n \in S_1 \cup S_2$. 
Then the Main Theorem, that is,
$\sum_{n \in \Z^d \cap K}  \prod_{i\in [t]} \tilde \tau(\psi_i(n)) 
= \vol(K) \prod_p \beta_p + o_{L,t,d}(N^d)$, holds if and only if
under the same conditions
$$\sum_{n \in \Z^d \cap K}  \prod_{i\in [t]} \bar \tau(\psi_i(n)) 
= \vol(K) \prod_p \beta_p + o_{L,t,d}(N^d)~.$$
\end{proposition}

\section{A majorant for the normalised divisor function}
\label{majorant-section}

Suppose that $A \subseteq [N]$ has cardinality $|A|=\delta N$.
Loosely speaking, if $0 < \delta < 1$ is fixed, we refer to such sets $A$,
for $N$ arbitrarily large, as \emph{dense}.
In this case, a sufficient condition for $A$ to contain approximately the
expected number of finite complexity structures is that $A$ is
sufficiently Gowers-uniform. This is to say, the uniformity norm 
$$
 \|1_A -\delta\|_{U^s[N]} 
 := \Bigg( 
    \E_{x \in [N]} 
    \E_{h\in[N]^s} 
    \prod_{\omega \in \{0,1\}^s} 
    (1_A -\delta)(x + \omega \cdot h ) 
    \Bigg)^{1/2^{s}}
$$ 
is small for some $s$ that is determined by the structure one is counting.
For instance, the number of $4$-term arithmetic progressions in a set $A$
of size $|A|=\delta N$ satisfies 
$$\E_{n+3d \leq N} 1_A(n)1_A(n+d)1_A(n+2d)1_A(n+3d) \sim \delta^4~,$$
if $\|1_A - \delta \|_{U^3}$ is small.
These results remain to be true when one replaces $1_A$ by a function
$f:\N \to \C$ that is bounded independent of $N$ and that has asymptotic
density $\E_{n\leq N}f(n) = \delta + o(1)$. 

If $f$ fails to satisfy these properties, that is, if it is
either sparse or unbounded, then a \emph{transference principle} is
required. 
Such a principle was established by Green and Tao in
\cite{green-tao-longprimeaps,green-tao-linearprimes}
and is based on the observation that a sparse set that is relatively
dense in a random-like set behaves in the same way as a dense set.

The first step is to replace the function $f$ by a model $\tilde f$ that
has asymptotic density. Examples are the replacement of the characteristic
function of primes by the von Mangoldt function or the replacement of
$\tau$ by ${\tilde \tau} $ in our case.

An application of the transference principle requires a 
\emph{majorant} function $\nu : [N] \rightarrow \C$ with 
$|\tilde f(n)| \leq C\nu(n)$ for all $n$ 
which satisfies the linear forms and correlation conditions of 
\cite[\S 6]{green-tao-linearprimes}, two conditions which are designed to
model a random measure. (We recall precise statements in Section
\ref{linear forms section}.)
This majorant replaces the ``random-like set'' from the observation.
The relative density condition from the observation is also present in the
generalised case. 
Indeed, part of the definition of $\nu$ is that 
$\E_{n \leq N} \nu(n) = 1 + o(1)$, and we further replaced the
original function $f$ by a dense model $\tilde f$. Thus we have
$$ \delta~ \E_{n \leq N} \nu(n) 
 = \delta(1+o(1)) 
 \leq (1+o(1)) \E_{n \leq N} \tilde f (n) 
 \leq C(1+o(1)) \E_{n \leq N} \nu(n) $$
and hence $\tilde f$ can be regarded as being `dense' in $\nu$.

The Koopman--von Neumann theorem
\cite[Prop.10.3]{green-tao-linearprimes}, 
or \cite[Prop.8.1]{green-tao-longprimeaps}, then provides a result
corresponding to the above observation: 
Any function $f$ with asymptotic density $\E f$ that is dominated by a
pseudorandom measure $|f(n)| \leq \nu(n)$ may be decomposed as a sum
$f=f_1+f_2$ where $f_1$ is bounded and $f_2 - \E f_2$ has small uniformity
norms. 
Thus, $f-\E f$ has small uniformity norms if and only if the bounded
function $f_1 - \E f_1$ has, and one can apply the results from the dense
setting to $f_1$. 
That is, we have `transferred' the problem to the dense setting, provided
there is a way to deal with the error $f_2 - \E f_2$. 
Such a way is provided by \cite[Cor. 11.6]{green-tao-linearprimes}.

In the case $f = \Lambda$, Green and Tao \cite{green-tao-longprimeaps}
construct, building upon work of Goldston and  Y{\i}ld{\i}r{\i}m,
the required pseudorandom majorant by modifying the majorant the proof of
Selberg's sieve is based on.
The key property of such a majorant resulting from a Selberg sieve is that
it has the form of a \emph{truncated} divisor sum
\[ \nu(n) := \sum_{d | n, d \leq N^{\gamma}} a_d\] for certain 
coefficients $a_d$ and where $\gamma> 0$ is a fixed constant that may be
chosen as small as necessary. 
Its importance lies in the fact that summing only over small divisors
ensures that the divisibility events that occur when checking the linear
forms condition are almost independent, which allows us to deduce
\emph{asymptotics} as required for the linear forms
condition.

Our aim in this section is to show that a majorant of similar structure
can be constructed in the case of the divisor function ${\tilde \tau} $.
A first attempt, given the above discussion, might be to take 
\[ \nu(n) = {\tilde \tau} _{\gamma}(n) 
   :=\frac{1}{\gamma \log N} \sum_{d | n : d \leq N^{\gamma}} 1~.\] 
Unfortunately, however, a result of Tenenbaum \cite[Cor.3]{tenenbaum}
asserts that if $\gamma < 1/2$ then for every $\lambda$ the majorant
condition ${\tilde \tau} (n) \leq \lambda {\tilde \tau} _{\gamma}(n)$
fails to hold on a
positive proportion of $n \in [N]$. 
A modification of this idea is therefore required.
It turns out that the proportion of such `bad' $n$ can be
bounded by $\lambda^{-c\log \log \lambda}$ for some $c = c(\gamma) > 0$.
Denoting by $X(\lambda)$ the set of bad $n$ for $\lambda$, then the bound
on $|X(\lambda)|$ allows us to sum $\sum_{i\geq 1} \lambda_i
1_{X(\lambda_i)}(n) {\tilde \tau} _{\gamma} (n)$ for suitable sequences
$(\lambda_i)$.

The idea behind this is due to Erd\H{o}s \cite{erdos}: 
Let $N^{\gamma} < n \leq N$. 
Considering the distribution of prime factors of such a number, one
expects that ${\tilde \tau} (n)$ is essentially controlled by the number
of small divisors ${\tilde \tau} _{\gamma}(n)$. 
But when is this actually the case?
A sufficient condition may be obtained as follows.
Write $n = p_1^{a_1} \dots p_t^{a_t}$, where the primes are ordered by
increasing size, and let $p_1^{a_1} \dots p_{j+1}^{a_{j+1}}$ be the first
initial partial product that exceeds $N^{\gamma}$.
Then we are guaranteed control of ${\tilde \tau} (n)$ by ${\tilde \tau}
_{\gamma}(n)$ provided
$p_{j+1}$ is large,
since $n$ has at most $\frac{\log N}{\log p_{j+1}}$ prime factors
$>p_{j+1}$. The quality of control depends on the size of $p_{j+1}$.
Suppose $n$ is a `bad' integer for which the control is of ${\tilde \tau}
(n)$ by
${\tilde \tau} _{\gamma}(n)$ is not good enough, thus, $p_{j+1}$ is quite
small. 
The smaller $p_{j+1}$ is, the worse is the control, but, also, the denser
gets the distribution of prime factors of the large initial product of
$n$. 
Excluding the sparse set of numbers that have a large proper prime power
divisor, one expects to find some structure in the `dense' set of prime
factors $<p_{j+1}$.
A pigeonhole argument shows that there is some short interval that
contains quite a large number of those prime factors $<p_{j+1}$; a very
sparse event.

The prime divisor structure of `bad' integers $n$ that this proof
strategy provides will be important later on, because it allows us to
\emph{explicitly} describe the exceptional set for the inequality 
${\tilde \tau}(n) \leq \lambda {\tilde \tau} _{\gamma}(n)$ at level
$\lambda$.

The following lemma is a reformulation of Erd\H{o}s's observations from
\cite{erdos}.

\begin{lemma}[Erd\H{o}s]\label{erdos-lemma}
Let $n \leq N$ and suppose that ${\tilde \tau} (n) \geq 2^s{\tilde \tau}
_{\gamma}(n)$ for
some $s > 2/\gamma$. Then one of the following three alternatives holds:
\begin{enumerate}
\item $n$ is excessively ``rough'' in the sense that it is divisible by 
some prime power $p^a$, $a \geq 2$, with $p^a > \log^{C_1} N~;$ 
\item $n$ is excessively ``smooth'' in the sense that if 
$n = \prod_{p} p^a$ then
\[ \prod_{p \leq N^{1/(\log \log N)^3}} p^a \geq N^{\gamma/\log\log N}~;\]
\item $n$ has a ``cluster'' of prime factors in the sense that there is an
$i$,  $\log_2 s - 2 \leq i \ll \log \log \log N$ such that $n$ has at
least $\gamma s(i + 3 - \log_2 s)/100$ prime factors in the superdyadic
range $I_i := [N^{1/2^{i+1}}, N^{1/2^i}]$ and is not divisible by the
square of any prime in this range.
\end{enumerate}
\end{lemma}
\begin{proof} 
The alternatives (i) and (ii) correspond to the sets $S_1$ and $S_2$
from Section \ref{section: arithmetical lemmas} and thus can be regarded
as \emph{exceptional}. 
Suppose that $n$ is unexceptional, that is (i) and (ii) are not satisfied,
and that the prime factorisation of $n$ is given by
\[ n = p_1^{a_1} \dots p_k^{a_k},\] 
where $p_1 < \dots < p_k$. 
Let $j$ be the index for which
\begin{equation}\label{j-definition} 
p_1^{a_1} \dots p_j^{a_j} \leq N^{\gamma} 
 < p_1^{a_1} \dots p_{j+1}^{a_{j+1}},
\end{equation}and write
$$ n' := p_1^{a_1} \dots p_j^{a_j}~.$$
We claim that $n' \geq N^{\gamma/2}$. Indeed, if this is not the case,
then $p_{j+1}^{a_{j+1}} \geq N^{\gamma/2}$. 
Since (i) does not hold we have $a_{j+1} = 1$. 
Thus, since $p_{j+1} \dots p_k | n$, we have $k - j \leq 2/\gamma$. 
Furthermore, using the fact that (i) does not hold once more,
we have $a_{j+1} = \dots = a_k = 1$ and so in this case
$$ {\tilde \tau} (n)
= 2^{k-j} {\tilde \tau} (n') \leq 2^{2/\gamma}{\tilde \tau} (n') 
\leq 2^{2/\gamma} {\tilde \tau} _{\gamma}(n) 
< 2^s {\tilde \tau}_{\gamma}(n)~, $$
contrary to assumption. 
Let $r \geq 1$ be the unique integer such that 
$$ N^{\gamma/(r+1)} <  p_j \leq N^{\gamma/r}~.$$ 
Then
$$ a_{j+1} + \dots + a_k \leq \frac{\log N}{ \log p_j} \leq
   \frac{r+1}{\gamma}~, $$ 
which means that 
$$ {\tilde \tau} (n) 
= (a_{j+1} + 1) \dots (a_k + 1) {\tilde \tau} (n') 
\leq 2^{a_{j+1} + \dots + a_k}{\tilde \tau} (n') 
\leq 2^{(r+1)/\gamma}{\tilde \tau}_{\gamma}(n) 
\leq 2^{2r/\gamma}{\tilde \tau} _{\gamma}(n)$$
and thus, recalling
the assumption $ 2^s{\tilde \tau} _{\gamma}(n) \leq {\tilde \tau} (n)$,
we have
$$ r \geq s\gamma/2.$$
All prime factors of $n'$ are therefore bounded by $N^{2/s}$.

Since we are not in the exceptional case (ii), the small prime factors
have a negligible contribution
\begin{equation}\label{a-bound} 
\prod_{p \leq N^{1/(\log \log N)^3}} 
p^a \leq N^{\gamma / \log \log N}~.
\end{equation}

Consider the smallest collection of superdyadic intervals
$I_i = [N^{1/2^{i+1}}, N^{1/2^i}]$ which cover 
$(N^{1/(\log \log N)^3},N^{2/s}]$; hence, these $i$ satisfy 
$\log_2 s - 2 \leq i < 6 \log \log \log N$.
In view of \eqref{a-bound}, the bound $p_j \leq N^{2/s}$ and the
fact that $n' \geq N^{\gamma/2}$, we obtain
$$ \prod_i  \prod_{\substack{p \in I_i\\ p^a \Vert n}} p^a 
    \geq N^{\gamma/2 - \gamma/\log \log N} > N^{\gamma/4}~.$$
Since $n$ is unexceptional (and, specifically, (i) does not hold), all of 
the $a$'s appearing here are equal to one.
Thus if the lemma were false, we would have
$$ N^{\gamma/4} \leq
 \prod_{i \geq \log_2 s -2} N^{\gamma s(i + 3 - \log_2 s)/(100 \cdot 2^i)}
 = \exp \left(\log N  \sum_{j \geq 1} \gamma \frac{j}{2^j}
\frac{s}{2^{\log_2 s -2}}\frac{1}{100}\right)
< N^{\gamma/4} ,$$ 
a contradiction\footnote{The somewhat arbitrary factor of
$100$ could have been replaced by any other positive number that was large
enough to induce this contradiction.}.
\end{proof}

It is possible to bound the number of $n \leq N$ satisfying condition
(iii) for some value of $i$ just using their specific structure. 
Setting $m_0 := \lceil \gamma s(i + 3 - \log_2 s)/100 \rceil$, 
write $X(i,s)$ for the set of $n \leq N$ divisible by at least $m_0(i,s)$
primes in $[N^{1/2^{i+1}},N^{1/2^i}]$. 
Thus
\[ N^{-1} |X(i,s)| 
\leq \frac{1}{m_0!} \left( \sum_{p \in I_i} \frac{1}{p} \right)^{m_0} =
\frac{1}{{m_0}!} (\log 2 + o(1))^{m_0}~.\] 
The crude bound $m! \geq \left(\frac{m}{e}\right)^{m}$ yields
the estimate
\begin{equation}\label{xim-bound} 
  N^{-1} |X(i,s)| 
 \leq \left\{
\begin{array}{ll} 
(c/\gamma s)^{\gamma s} & \mbox{if $s/4 \leq 2^i \leq s^2$} \\ 
(c/\gamma s)^{\gamma s i} & \mbox{if $2^i > s^2$}~,
\end{array} \right.
\end{equation}
and hence
\begin{equation}\label{xim-bound-2} 
N^{-1} \sum_{i \geq \log_2 s - 2} 
|X(i,s)| \leq (c/\gamma s)^{\gamma s} \log_2 s~.
\end{equation}
In particular, given the paucity of integers $n$ satisfying (i) and 
(ii) as guaranteed by Lemma \ref{paucity},
this together with Lemma \ref{erdos-lemma} shows that the density of $n
\leq N$
for which  ${\tilde \tau} (n) > 2^s {\tilde \tau}_{\gamma}(n)$ is
bounded by $2^{-c_{\gamma} s\log s}$. 
The fast decay of these densities makes the following definition
reasonable.

\begin{proposition}[Majorant for the divisor function] \label{majorant}
Fix $\gamma > 0$. Write $U(i,s)$ for the set of all products of 
$m_0(i,s):= \lceil \gamma s(i + 3 - \log_2 s)/100 \rceil$ distinct primes
from the interval $[N^{1/2^{i+1}}, N^{1/2^i}]$. 
Define $\nu : [N] \rightarrow \R_+$ by 
\[ 
 C\nu(n) :=
 2^{2/\gamma} {\tilde \tau} _{\gamma}(n) + 
 \sum_{s > 2/\gamma}^{(\log \log N)^3}
 \sum_{i = \log_2 s - 2}^{6 \log \log \log N}
 \sum_{u \in U(i,s)}
 2^s 1_{u|n} {\tilde \tau}_{\gamma}(n) +
 1_{n \in S_1 \cup S_2} {\tilde \tau}(n)~,
\] 
where $S_1 \cup S_2$ is the set of all $n \leq N$
satisfying either (i) or (ii) of Lemma \ref{erdos-lemma}.
Then there is a value of $C$ (depending on $\gamma$) such that 
$\E_{n \leq N} \nu (n) = 1 + o(1)$.
For all $n \leq N$ we have ${\tilde \tau}(n) \leq C\nu(n)$.
\end{proposition}
\begin{remarks}
\emph{(1)}\quad
Since $\gamma$ will be as small as necessary in every later application,
we may as well choose it to be the reciprocal of an integer. 
This has the advantage that, setting $U(i, 2/\gamma):=\{1\}$ for 
$i = \log_2 s - 2$ and $U(i, 2/\gamma):=\emptyset$ otherwise, we can write
$$
 C\nu(n) =  
 \sum_{s = 2/\gamma}^{(\log \log N)^3}
 \sum_{i = \log_2 s - 2}^{6 \log \log \log N}
 \sum_{u \in U(i,s)}
 2^s 1_{u|n} {\tilde \tau}_{\gamma}(n) +
 1_{n \in S_1 \cup S_2} {\tilde \tau}(n)~.$$
\emph{(2)}\quad
While $\nu$ can be shown to be pseudorandom, a further reduction in the
next section will allow us to save some work by dropping the exceptional
term $1_{n \in S_1 \cup S_2} {\tilde \tau}(n)$.

\noindent
\emph{(3)}\quad
Finally, note that the divisors $u \in U(i,s)$ are truncated divisors
themselves, that is, they satisfy $u \leq N^{\gamma}$. 
Indeed, suppose $i + 3 - \log_2 s = j (> 1)$, and hence $s/2^i = 8/ 2^j$,
then 
$$u 
\leq N^{m_0(i,s)/2^{i}} 
\leq N^{2 \gamma s j /(100 \cdot 2^{i})}
\leq N^{2 \gamma 8 j /(100 \cdot 2^{j})} 
< N^{\gamma}~.$$
\end{remarks}

\begin{proof}
The fact that ${\tilde \tau} (n) \leq C \nu(n)$ is an immediate
consequence of Lemma \ref{erdos-lemma}. 
To show the existence of $C$, we have to check that the expectation of
$\nu$ on the integers $\leq N$ is bounded independent of $N$. 
Note that
\begin{align*}
 \E_{n\leq N} \sum_{u \in U(i,s)} 1_{u|n}
     {\tilde \tau} _{\gamma}(n) 
 \leq \frac{1}{m_0!} \Bigg( \sum_{p \in I_i} \frac{1}{p} \Bigg)^{m_0}
     \frac{1}{\gamma \log N } \sum_{m \leq N^{\gamma}} \frac{1}{m}
 \leq \frac{1}{m_0!} (\log 2 + o(1))^{m_0} ~.
\end{align*}
This allows us to make use of a bound of type \eqref{xim-bound-2}. In
detail,
\begin{align*}
&\E_{n\leq N} 
 \sum_{s \geq 2/\gamma} 
 \sum_{i \geq \log_2 s - 2}
 \sum_{u \in U(i,s)} 
  1_{u|n} {\tilde \tau}_{\gamma}(n) 2^s \\
&\leq \sum_{s \geq 2/\gamma} \sum_{i \geq \log_2 s - 2} 
  \frac{1}{m_0!} (\log 2 + o(1))^{m_0} ~2^s \\
&\leq \sum_{s \geq 2/\gamma} \sum_{j \geq 1} 
  \left(\frac{100 \cdot e \cdot (\log 2 + o(1))}{\gamma s j}
  \right)^{\gamma sj/100} 2^s \\
&\leq \sum_{s \geq 2/\gamma}  \frac{2^s}{s^{s\gamma/100}} 
  \left( 
  \sum_{j \geq 1} \left(\frac{100 \cdot e \cdot (\log 2 + o(1))}{\gamma j}
  \right)^{\gamma j/100} \right)^s
\end{align*}
which converges. We note for later reference that the above
expression still converges when the factor $2^s$ is replaced by $a^s$
with any positive constant $a$.
\end{proof}

\section{$W$-trick}\label{W-trick-section}

The nilpotent Hardy-Littlewood method employs the uniformity of a
function to deduce an asymptotic for finite complexity correlations.
However, the divisor function ${\tilde \tau}$ is not equidistributed in
residue classes to small moduli and thus in particular not
Gowers-uniform. 
To remove this obstruction, we shall use a so-called $W$-trick and
decompose ${\tilde \tau}$ into a sum of functions which do not detect a
difference between these residue classes.
This decomposition of ${\tilde \tau}$ can be viewed as a factorisation as
product of a uniform function and an almost periodic function.

It is natural to consider the restricted divisor function that does not
count divisors with small prime factors at all:
\begin{definition}[$W$-tricked divisor function]
Set $w(N) := \frac{1}{2} \log \log N$ and $W:=\prod_{p<w(N)}p$.
We define $W$-tricked versions of ${\tilde \tau}$ and 
${\tilde \tau}_{\gamma}$ by
$${\tilde \tau}'(n) 
 := \frac{W}{\phi(W)} (\log N)^{-1} \sum_{(d,W)=1} 1_{d|n}~,$$
and
$${\tilde \tau}'_{\gamma}(n) 
 := \frac{W}{\phi(W)} (\gamma \log N)^{-1}
    \sum_{\substack{d \leq N^{\gamma}\\(d,W)=1}} 1_{d|n}~,$$
where $\phi$ denotes Euler's totient function.
\end{definition}
Thus, $\tilde \tau$ decomposes as a product
$$\tilde \tau(n) 
 = \tilde\tau'(n) \Bigg( \frac{\phi(W)}{W}
   \sum_{\substack{w \in \N~: \\ p|w~ \Rightarrow~ p<w(N)}}
   1_{w|n} \Bigg)~,$$
where the first factor is expected to be uniform and the second factor is
almost periodic.
We may, in fact, replace the second factor by a periodic function: 
Setting 
$$\W 
:= \prod_{p \leq w(N)} 
  p^{\left\lfloor C_1 \log_p(\log N) \right \rfloor}
\leq (\log N)^{C_1 \pi(w(N))}
\ll \exp\left(\frac{C_1(\log\log N)^2}{\log\log\log N}\right)~,
$$
define the following explicit function $\bar \tau : \Z \to \R$ by 
$$\bar \tau(n) 
 := \tilde\tau'(n)
 \Big( \frac{\phi(W)}{W} 1_{\varpi(n)|(\W/W)} \sum_{w | \W} 1_{w|n}
\Big)~,$$
where we denote by
$$\varpi(n):= \prod_{p^a\| n, p\leq w(n)}p^a$$
the $w(N)$-smooth factor of $n$.
Any integer $n$ that does not satisfy $\varpi(n)|(\W/W)$ is divisible by
some
proper prime power $p^a > \log^{C_1}N$, and hence belongs to the
exceptional set $S_1$ (c.f.~Lemma \ref{(ii)-bound}). 
Thus, $\bar \tau$ satisfies the conditions of Proposition
\ref{tau-bar-reduction}.
This will allow us to deduce the Main Theorem from the following
proposition, to be established in Section \ref{application_section}.

\begin{proposition}\label{w-reduction2}
Let $M=N/\W$, let $\tilde\Psi$ be a finite complexity system of
affine-linear forms whose linear coefficients are bounded by $L$.
Then for any choice of $b_1,\dots, b_t \in [\W]$ such that 
$\varpi(b_i)| (\W/W)$ for all $i \in [t]$,
\begin{equation*}
\E_{n \in \Z^d \cap K'}
\prod_{i \in [t]} {\tilde \tau}'(\W \tilde\psi_i(n) + b_i) 
= 1 + o_{d,t,L}(M^d/\vol(K'))
\end{equation*}
holds for every convex body $K' \subseteq [-M,M]^d$ which satisfies 
$\W \tilde\Psi(K')+(b_1,\dots,b_t)\subseteq[1,N]^d$.
\end{proposition}

\begin{proof}[Proof of the Main Theorem from Proposition
\ref{w-reduction2}]
Assume $K$ and $\Psi$ satisfy the conditions of the Main Theorem.
Fix some $a \in [\W]^d$ and let $\tilde \psi_{i,a} : \Z^d \to \Z$ be the
affine-linear function for which
$$\psi_i(\W n+a) = \W \tilde \psi_{i,a}(n) + b_i(a)~,$$
where $b_i(a) \in [\W]$. 
Note that $\psi_i$ and $\tilde \psi_{i,a}$ only differ in
the constant term. 

Define
$K'_a \subset \R^d$ to be the convex body 
$\{x \in \R^d: \W x+a \in K\}$ and note that 
$\vol(K'_a) = \vol(K)/\W^d$.
By Proposition \ref{w-reduction2}, we then have
\begin{align*}
&\sum_{\substack{n \in \\ \Z^d \cap K}} 
 \prod_{i \in [t]} {\bar \tau} (\psi_i(n))\\
&= \sum_{a \in [\W]^d}
 \sum_{\substack{\W n + a  \\ \in \Z^d \cap K }}
 \prod_{i \in [t]} 
 \Bigg( \frac{\phi(W)}{W} \sum_{w|\W} 1_{w|\psi_i(a)}
 \Bigg)
 1_{\varpi(\psi_i(a))|(\W/W)}
 {\tilde \tau}'(\W \tilde \psi_{i,a}(n) + b_i(a)) \\
&=(1+ o_{t,d,L}(1)) 
 \frac{\vol K}{\W^d}
 \sum_{a \in [\W]^d} 
 \prod_{i \in [t]} 1_{\varpi(\psi_i(a))|(\W/W)}
 \Bigg( \frac{\phi(W)}{W} \sum_{w|\W} 1_{w|\psi_i(a)} 
  \Bigg)  \\
&=(1+ o_{t,d,L}(1)) 
 \vol K \frac{\phi(W)^t}{W^t}
 \E_{a \in [\W]^d}
 \prod_{i \in [t]} 1_{\varpi(\psi_i(a))|(\W/W)}
 \Bigg(  \sum_{w|\W} 1_{w|\psi_i(a)} 
  \Bigg)~.
\end{align*}
The latter expectation can be expressed in terms of local divisor
densities: 

Define $\alpha(p)$ for $p < w(N)$ such that
$\W = \prod_{p < w(N)}p^{\alpha(p)}$. 
Then
\begin{align*}
&\E_{a \in [\W]^d}
 \prod_{i \in [t]} 
 \sum_{w|\W} 1_{w|\psi_i(a)} 
  1_{\varpi(\psi_i(a))|(\W/W)} \\
&= \prod_{p < w(N)}
  \sum_{e_1, \dots, e_t < \alpha(p)} 
  \E_{n \in (\Z/p^{\alpha(p)}\Z)^{d}} \prod_{i \in[t]} 
  \Big(1_{\psi_i(n)\equiv 0 \Mod{p^{e_i}}} -
  1_{\psi_i(n)\equiv 0 \Mod{p^{\alpha(p)}}}\Big) \\
&= \prod_{p < w(N)} \Bigg(
  \sum_{e_1, \dots, e_t \in \N} 
  \E_{n \in (\Z/p^{\max_j e_j}\Z)^{d}} \prod_{i \in[t]} 
  1_{\psi_i(n)\equiv 0 \Mod{p^{e_i}}}\\
& \quad \quad \quad + O\bigg(
     \sum_{\substack{e_1, \dots, e_t \\ < \alpha(p)}}p^{-\alpha(p)}
     + \sum_{\substack{e_1, \dots, e_t \\ \max_j e_j \geq \alpha(p) }}
        p^{-\max_j a_j} \bigg) 
  \Bigg)~.
\end{align*}
Since $\alpha(p) = C_1 \frac{\log \log N}{\log p} + O(1)$, we may bound
the error term by
\begin{align*}
\sum_{\substack{e_1, \dots, e_t \\ < \alpha(p)}} 
  p^{-\alpha(p)}
+ \sum_{\substack{e_1, \dots, e_t \\ \max_j e_j \geq \alpha(p) }}
  p^{-\max_j a_j}
\ll \frac{(\log \log N)^t}{(\log N)^{C_1}} 
  + \sum_{j\geq C_1 \frac{\log\log N}{\log p}} p^{-j} j^t 
 \ll (\log N)^{-C_1 + 1}~.
\end{align*}
Hence
\begin{align*}
 \E_{a \in [\W]^d}
 \prod_{i \in [t]} 
 \sum_{w|\W} 1_{w|\psi_i(a)} 
  1_{\varpi(\psi_i(a))|\W}
&= \prod_{p < w(N)} 
  (\sum_{e_1, \dots, e_t \in \N} 
   \alpha(p^{e_1},\dots, p^{e_t}) + (\log N)^{-C_1 +1})\\
&= \prod_{p < w(N)} 
  (\beta_p(1-p^{-1})^t + (\log N)^{-C_1 +1}) \\
&= (1+o(1))\prod_{p < w(N)} 
   \beta_p(1-p^{-1})^t~,
\end{align*}
where the last step follows, keeping in mind that $1 \leq \beta_p \ll 1$,
from
$$\sum_{p\leq w(N)}\log (1 + O((\log N)^{-C_1 +1})) 
\ll w(N)(\log N)^{-C_1 +1} \ll \frac{\log \log N}{(\log N)^{C_1 -1}}~.$$
Since
$ \prod_{p \leq w(N)}(1-p^{-1})^{-1}=\frac{W}{\phi(W)}~, $
the above implies
$$ 
 \E_{n \in \Z^d \cap K} 
  \prod_{i \in [t]} {\bar \tau} (\psi_i(n)) 
 = (1+o(1))\prod_{p \leq w(N)} \beta_p ~.$$
The local factors bound \eqref{beta_p-bound}, that is
$\beta_p= 1+O_{t,d,L}(p^{-2})~,$
and Proposition \ref{tau-bar-reduction} yield the Main Theorem:
$$ \E_{n \in \Z^d \cap K} \prod_{i \in [t]} {\tilde \tau} (\psi_i(n))
 = \prod_{p} \beta_p + o(1)~.$$
\end{proof}

\stsubsection{$W$-tricked majorant}
In order to prove Proposition \ref{w-reduction2}, we require for any given
choice of $b=(b_1, \dots, b_t) \in [\W]^t$ a majorant that simultaneously
majorises all of the functions
$n \mapsto {\tilde \tau}'(\W n + b_i)$ for $i=1,\dots,t$.
Define 
\begin{align} \label{nu-dashed}
C'\nu'(n) 
:= 
 \sum_{s = 2/\gamma}^{(\log \log N)^3}
 \sum_{i = \log_2 s - 2}^{6 \log \log \log N}
 \sum_{u \in U(i,s)} 
 2^s 1_{u|n}{\tilde \tau}'_{\gamma}(n)~,
\end{align}
where $C'$ is such that $\E_{n \leq N} \nu'(n) = 1+o(1)$.
This and the definition of $\tilde \tau'$ imply 
$\E_{n \leq M} \nu'(\W n + a) = 1+o(1)$ for all $a \in [\W]$.

Thus, a majorant of the required form is given by a constant multiple of
$$\nu'_{\W,b}:\Z^t \to \Z, 
\qquad 
\nu'_{\W,b} := \E_{i \in [t]} \nu'(\W n+b_i)~.$$ 
Note furthermore that $\nu'_{\W,b}$ still satisfies the condition
$\E_{m \leq M} \nu'_{\W,b}(m) = 1 + o(1)$.

\section{The linear forms condition}
\label{linear forms section}

The aim of the following two sections is to show that the following slight
modification of the majorant 
$\nu'_{\W,(b_i)} = \E_{i \in [t]} \nu'(\W n+b_i)$ is indeed pseudorandom. 
Let $M'$ be a prime satisfying $M < M' \leq O_{t,d,L}(M)$ and define
$\nu^{*}_{\W,(b_i)}:[M'] \to \R^+$ by
\begin{align*}
 \nu^*_{\W,(b_i)}(n)= 
 \left\{
 \begin{array}{ll}
 \frac{1}{2}(1+\nu'_{\W,(b_i)}(n)) &\text{ if } n \leq M \cr
 1 & \text{ if } M < n \leq M'~.
 \end{array}
 \right.
\end{align*}
As is seen in \cite[App.D]{green-tao-linearprimes},
$\nu^{*}_{w,(b_i)}$ is $D$-pseudorandom if it satisfies the following two
propositions, which are technical reductions of the linear forms and
correlation conditions from \cite{green-tao-linearprimes}. 

\begin{proposition}[$D$-Linear forms estimate]
\label{linear forms estimate}
Let $1 \leq d,t \leq D$ and 
let $(i_1, \dots, i_t) \in [t]^t$ be an arbitrary collection of indices.
For any finite complexity system
$\Psi : \Z^d \to \Z^t$ whose linear coefficients are bounded by $D$
and for every convex body $K \subseteq [-M,M]^d$ which satisfies 
$$\W \Psi(K) +(b_{i_1},\dots,b_{i_t}) \subseteq [1,N]^t~,$$ the asymptotic
$$\E_{n\in \Z^d \cap K } 
  \prod_{j \in [t]} 
  \nu'(\W \psi_j(n) + b_{i_j}) 
= 1 + O_D(N^{d - 1 +O_D(\gamma)}/\vol(K)) + o_D(1)$$
holds, provided $\gamma$ was small enough.
\end{proposition}

\begin{proposition}[Correlation estimate]
\label{verification of C-Condition}
For every $1 < m_0 \leq D$ there exists a function $\sigma_{m_0} : \Z_{M'}
\to \R^+$ with bounded moments 
$\E_{n \in \Z_{M'}} \sigma_{m_0}^q(n) \ll_{m,q} 1$
such that 
$$\E_{n \in I} \prod_{j \in [m]} \nu'(\W(n+h_j)+b_{i_j}) 
  \leq \sum_{1 \leq i < j \leq m} \sigma_{m_0}(h_i - h_j)~$$ 
holds for every interval $I \subset \Z_{M'}$, for every $1 \leq m \leq
m_0$ and every $m$-tuple $(i_1, \dots, i_{m}) \in [t]^{m}$, and for every
choice of (not necessarily distinct) $h_1,\dots,h_{m} \in \Z_{M'}$,
provided $\gamma$ was small enough.
\end{proposition}

The correlation estimate will be deferred to the next section, the
verification of the linear forms condition is an immediate consequence of
the following proposition.

\begin{proposition}\label{linear-forms-prop}
Let $1 \leq d,t \leq D$ and let $\Psi:\Z^d \to \Z^t$ be a system of
affine-linear forms, such that any exceptional prime, that is, any prime
$p$ for which there are $\psi_i$ and $\psi_j$ that are affinely related
modulo $p$, satisfies $p \leq w(N)$.
(Observe that we make no assumption on the coefficients of $\dot\Psi$.)
Then
$$\E_{n\in \Z^d \cap K} \prod_{j \in [t]} \nu'(\psi_j(n)) 
= 1 + O_D(N^{d - 1 +O_D(\gamma)}/\vol(K)) + o_D(1)$$
for every convex body $K \subseteq [-N,N]^d$ such that
$\Psi(K)\subseteq[1,N]^t$.
\end{proposition}

\begin{proof}[Proof of Proposition \ref{linear forms estimate}]
The system $\Psi$ of affine-linear forms that appears in the linear forms
condition has the property that no two forms $\psi_i$, $\psi_j$ are
affinely related and that the coefficients of $\dot \Psi$ are bounded by
$D$. 
Thus every exceptional prime $p$ of $\Psi$ satisfies $p=O_{D}(1)$.
We have to show that
$$ \E_{n\in \Z^d \cap K} \prod_{j \in [t]} \nu'(\phi_j(n)) 
 = 1+ O_D(N^{d - 1 +O_D(\gamma)}/\vol(K)) + o_D(1) $$
with $$\phi_j(n) = \W \psi_j(n)+b_{i_j}~.$$
If $p>w(N)$ is a prime, then $\phi_i$ and $\phi_j$ are affinely
related modulo $p$ if and only if $\psi_i$ and $\psi_j$ are
affinely related modulo $p$, which proves the result in view of
Proposition \ref{linear-forms-prop}. 
\end{proof}

\subsection*{Proof of Proposition \ref{linear-forms-prop}}
The strategy of the proof is to show that all occurring dependent
divisibility events $\prod_{j\in [t]} 1_{a_i|\psi_i(n)}$ where the $a_i$
are not pairwise coprime have a negligible contribution. 
Removing those, the densities of the remaining events will depend on the
respective choice of $a_1,\dots, a_t$ but are, up to a small error,
independent of the $\psi_i$.

Recalling the definition \eqref{nu-dashed} of $\nu'$, our task is to show
that
\begin{align*}
\E_{n\in \Z^d \cap K} \prod_{j \in [t]} 
 \Bigg( 
   \sum_{s = 2/\gamma}^{(\log \log N)^3}
   2^s
   \sum_{\substack{i =  \log_2 s - 2}}^{6 \log \log \log N}
   \sum_{u_j \in U(i,s)}
   1_{u_j | \psi_j(n)} &{\tilde \tau}'_{\gamma}(\psi_j(n)) 
 \Bigg) \\
=&~C'^t + O_D\bigg( \frac{ N^{d - 1 +O_D(\gamma)}}{\vol(K)}\bigg) +
  o_D(1).
\end{align*}
An arbitrary cross term that appears when multiplying out is of the form
\begin{align}\label{cross term} 
 \E_{n \in \Z^d \cap K} 
 \prod_{j \in [t]}
 \sum_{u_j \in U(i_j,s_j)} 
 2^{s_j} {\tilde \tau}'_{\gamma}(\psi_j(n)) 
 1_{u_j|(\psi_j(n))}~. 
\end{align}
The sets $U(i,s)$ were defined in the statement of Proposition
\ref{majorant}. We will make use of two of their properties, namely that
any prime divisor $p$ of $u \in U(i,s)$ satisfies 
$p \gg N^{1/(\log\log N)^3}$ and that $u \leq N^{\gamma}$ for 
$u \in U(i,s)$.

The removal of dependent divisibility events will be carried out in a
sequence of steps. The first is the following claim.
\begin{claim}
The cross term \eqref{cross term} equals
\begin{align} \label{cross term 2}
 &\dsum_{u_1,\dots, u_t} 
 \E_{n \in \Z^d \cap K}~
 \prod_{j\in [t]} 
 2^{s_j} 1_{u_j|\psi_j(n)} 
 \frac{W}{\phi(W) \gamma \log N}
 \sum_{v_j|u_j} 
 \sum_{\substack{d_j \leq N^{\gamma}/v_j\\ (d_j, u_j W )=1}}
 1_{d_j|\psi_j(n)} \\ \nonumber
 &\qquad+ O_{D}(N^{-(\log \log N)^{-3}/4}) ~, 
\end{align}
where the notation $\dsum_{u_1, \dots, u_t}$ indicates that the summation
is extended only over pairwise coprime choices of $u_1, \dots, u_t$,
where $u_j \in U(i_j,s_j)$ for each $j$.
\end{claim}
\begin{remark}
 Since the sums over $s_j$ and $i_j$ only have $O_D((\log \log N)^4)$
terms, the contribution of error terms from all cross terms is
bounded by $O_D(N^{-(\log \log N)^{-3}/8}) = o_D(1)$.
\end{remark}

\begin{proof}Recalling the definition of $\tilde \tau'_{\gamma}$, we see
that all we
have to do is, firstly, to bound the contribution of non-coprime
choices of $u_1, \dots, u_t$ to \eqref{cross term}, and, secondly, to
bound the contribution of such $n \in \Z^d \cap K$ to \eqref{cross term}
for which 
$(u_j, \psi_j(n)/u_j)>0$ for some $j$.
Observe that whenever $(u_i,u_j)>1$, there is some $p \gg N^{1/(\log \log
N)^3}$ such that $p^2| \prod_{i\in[t]} \psi_i(n)$.
We also find $p \gg N^{1/(\log \log N)^3}$ such that 
$p^2| \prod_{i\in[t]} \psi_i(n)$ when $n$ satisfies 
$(u_j, \psi_j(n)/u_j)>0$.
By the properties of the function $\alpha$, in particular by
\eqref{alpha-bound}, we have 
\begin{align*}
 \sum_{\substack{N^{(\log \log N)^{-3}}\\<p<N^{\gamma}}} 
 \E_{n \in \Z^d \cap K} 1_{p^2|\prod_{i}\psi_i(n)} \ll_t
 \sum_{\substack{N^{(\log \log N)^{-3}}\\<p<N^{\gamma}}} p^{-2} 
 = O_t\left(N^{-(\log\log N)^{-3}}\right)~.
\end{align*}
Cauchy-Schwarz yields
\begin{align*}
&|\eqref{cross term}-\eqref{cross term 2}| <
\sum_{\substack{u_1,\dots, u_t\\(u_i,u_j)>1}}
 \E_{n \in \Z^d \cap K}~
 \prod_{j\in [t]} 
 2^{s_j}  {\tilde \tau}'_{\gamma}(\psi_j(n)) 1_{u_j|\psi_j(n)}\\
& \qquad \qquad \qquad \qquad \qquad+ \sum_{u_1,\dots, u_t}
 \E_{n \in \Z^d \cap K}~
 \prod_{j\in [t]} 
 2^{s_j}  {\tilde \tau}'_{\gamma}(\psi_j(n))
 1_{u_j|\psi_j(n)}
 1_{(u_j, \psi_j(n)/u_j)>1} \\
 & \ll  \Bigg( 
 \E_{n \in \Z^d \cap K}\!
 \sum_{\substack{N^{(\log \log N)^{-3}}\\ <p<N^{\gamma}}} 
 1_{p^2|\prod_{i}\psi_i(n)} 
 \Bigg)^{\frac{1}{2}} 
 \Bigg( \E_{n \in \Z^d \cap K}
 \prod_{j\in [t]}  2^{2s_j} 
 {\tilde \tau}'^2_{\gamma}(\psi_j(n)) 
 \Big(\!\sum_{\substack{u_j \in \\ U(i_j,s_j)}} 1_{u_j|\psi_j(n)}\Big)^2
 \Bigg)^{\frac{1}{2}}.
\end{align*}
The second factor may be bounded with the help of the $k$-th
moment estimate from Lemma \ref{divisor-bound} by
\begin{align*}
 \Big(\E_{n \in \Z^d \cap K}~ 2^{2 D (\log \log N)^3}
 \prod_{j\in [t]} 
 {\tau}^4(\psi_j(n)) \Big)^{1/2}
\ll_D 2^{D (\log \log N)^3} (\log N)^{O_D(1)}~.
\end{align*}
 This proves the claim since
$$(2^{(\log \log N)^3} \log N)^{O_{D}(1)} N^{-(\log \log N)^{-3}/2}
\ll N^{-(\log \log N)^{-3}/4}~.$$
\end{proof}

We proceed to analyse \eqref{cross term 2}.
To simplify the notation, fix any choice of integer tuples 
$u=(u_1, \dots u_t)$ and $v=(v_1, \dots, v_t)$ and let $\mathcal D_{u,v}$
be the set of all tuples $(d_1,\dots,d_t)$ satisfying $(d_i, u_i W) = 1$
and $d_i \leq N^{\gamma}/v_i$ for $i=1, \dots t$.
With this notation, we show the following.
\begin{claim} The main term of \eqref{cross term 2} satisfies
\begin{align}\nonumber
&\dsum_{u_1,\dots, u_t} 
 \E_{n \in \Z^d \cap K}~
 \prod_{j\in [t]} 
 2^{s_j} 1_{u_j|\psi_j(n)} 
 \frac{W}{\phi(W) \gamma \log N}
 \sum_{v_j|u_j} 
 \sum_{\substack{d_j \leq N^{\gamma}/v_j\\ (d_j, u_j W )=1}}
 1_{d_j|\psi_j(n)}\\
\label{main-term}
&= (1 + o_{D}(1))
 \left(\frac{W}{\phi(W)\gamma\log N}\right)^t
 \dsum_{u_1,\dots, u_t}\, \sum_{v_1|u_1,\dots,v_t|u_t}\,
 \dsum_{\substack{d_1,\dots ,d_t\\ \in \mathcal D_{u,v}}} 
 \prod_{j \in [t]} 
 \frac{2^{s_j}}{u_j} \frac{1}{d_j} \\
\nonumber
&\qquad + O(N^{d-1+O_t(\gamma)}/\vol(K))~.
\end{align}
\end{claim}
\begin{remark}
 Similar as with the previous claim, the fact that the sums over
$s_j$ and $i_j$ only have $O_D((\log \log N)^4)$
terms implies that the overall contribution of the error terms from here
is still $O(N^{d-1+O_t(\gamma)}/\vol(K))$.
\end{remark}
\begin{proof}
Multiplying by the normalising factor of $\tilde \tau'_{\gamma}$, and
applying the volume packing lemma
(Lemma \ref{volume-packing}), we have
\begin{align*} 
&\left( \frac{W}{\phi(W) \gamma \log N} \right)^{-t} 
\dsum_{u_1,\dots,u_t}
 \E_{n \in \Z^d \cap K}~
 \prod_{j\in [t]} 2^{s_j} 
 \sum_{v_j|u_j} 
 \sum_{\substack{d_j \leq N^{\gamma}/v_j \\ (d_j,u_jW)=1}}
 1_{d_ju_j|\psi_j(n)} \\
&= \dsum_{u_1,\dots,u_t} \prod_{j\in [t]} 2^{s_j} 
 \sum_{v_1|u_1, \dots, v_t|u_t}
 \sum_{\substack{d_1,\dots,d_t\\ \in \mathcal D_{u,v}}}  
 \Bigg\{ \frac{\alpha (d_1, \dots, d_t)}{u_1 \dots u_t}
+ O\Big(\frac{N^{d-1}}{\vol(K)} \lcm(u_1d_1,\dots, u_td_t)\Big)\Bigg\}~.
\end{align*}
The error term is of order $O(N^{d-1 + O_t(\gamma)}/\vol(K))$.
Its total contribution is also seen to be 
$O(N^{d-1 + O_t(\gamma)}/\vol(K))$, since $2^{s_j} \leq 2^{(\log \log
N)^3}$, since $W/(\phi(W) \gamma \log N) \ll 1$ and since the sums over
the $u_j$ and $v_j$ have altogether $N^{O_t(\gamma)}$ terms.

Concerning the main term, Lemma \ref{dependent_div_events} allows us to
also pass to only summing over pairwise coprime choices of 
$d_1,\dots d_t$: for a fixed choice of $u$ and $v$ the sum over
$d_1,\dots,d_t$ in the main term satisfies
\begin{align*}
\dsum_{(d_1,\dots,d_t) \in \mathcal D_{u,v} } \alpha (d_1, &\dots, d_t) 
\leq \sum_{(d_1,\dots,d_t) \in \mathcal D_{u,v} } \alpha (d_1, \dots, d_t)
\\
&\leq \dsum_{(d_1,\dots,d_t) \in \mathcal D_{u,v} } 
 \alpha (d_1, \dots, d_t)
 \prod_{\substack{p \nmid d_1 \dots d_t \\ p > w(N)}} 
 \Bigg(1 +  
 \sum_{\substack{a_1, \dots, a_t \\ \text{at least two } a_i \not= 0}} 
 \alpha(p^{a_1}, \dots, p^{a_t}) \Bigg) \\
&\leq \dsum_{(d_1,\dots,d_t) \in \mathcal D_{u,v}}
 \alpha (d_1, \dots, d_t)
 \prod_{p > w(N)} (1+ O_D(p^{-2}))\\
&\leq (1+ O_D(1/w(N)))
 \dsum_{(d_1,\dots,d_t) \in \mathcal D_{u,v}} \alpha (d_1, \dots, d_t) \\
&= (1+ O_D(1/w(N)))
 \dsum_{(d_1,\dots,d_t) \in \mathcal D_{u,v}} \frac{1}{d_1 \dots d_t}~,
\end{align*}
which implies the claim.
\end{proof}
The last remaining step will be to show that, picking up only another
$(1+o_{D}(1))$ factor, we can move the product over $j$ in front in
the term \eqref{main-term}.
\begin{claim} Summing all terms \eqref{main-term}, we have
\begin{align} \nonumber
&\sum_{s_1, \dots , s_t} \sum_{i_1, \dots, i_t}
 \left(\frac{W}{\phi(W)\gamma\log N}\right)^t
 \dsum_{u_1,\dots, u_t} \sum_{v_1|u_1,\dots,v_t|u_t}
 \dsum_{\substack{(d_1,\dots ,d_t)\\ \in \mathcal D_{u,v}}} 
 \prod_{j\in [t]} 
 \frac{2^{s_j}}{u_j} \frac{1}{d_j}
\\ \nonumber
&=(1+o_D(1))\times \\ \label{j-product}
&\qquad 
\prod_{j\in [t]} \sum_{s_j = 2/\gamma}^{(\log \log N)^3}
 \sum_{\substack{i_j = \log_2 s_j - 2}}^{6 \log \log \log N}
 \sum_{u_j \in U(i_j,s_j)} \frac{2^{s_j}}{u_j} 
 \frac{W}{\phi(W)\gamma \log N} 
 \sum_{v_j|u_j}
 \sum_{\substack{d_j \leq N^{\gamma}/v_j \\(d_j,u_jW)=1 }}\frac{1}{d_j}\\
\nonumber
& \quad + O_t(N^{-1/(\log\log N)^3})~. 
\end{align} 
\end{claim}
\begin{proof}
The new expression \eqref{j-product} includes additional terms containing
non-coprime tuples $u_1,\dots,u_t$ or $d_1,\dots,d_t$. 
To see that these terms only contribute an additional
$(1+o_D(1))$ factor, first consider the $d_j$'s:
Note that 
$$\prod_{j \in [t]} \frac{1}{d_j} \leq \alpha(d_1, \dots, d_t)~.$$
Thus, an application of Lemma \ref{dependent_div_events}, similar to the
one for the previous claim, yields
$$ 
  \dsum_{(d_1,\dots,d_t) \in \mathcal D_{u,v}} 
   \prod_{j\in [t]} \frac{2^{s_j}}{u_j} \frac{1}{d_j}
 = (1+o_{D}(1))
  \sum_{(d_1,\dots,d_t) \in \mathcal D_{u,v}} 
   \prod_{j\in [t]} \frac{2^{s_j}}{u_j} \frac{1}{d_j}~. 
$$
 
It remains to show that we can also drop the coprimality condition
on the $u_j$'s. 
The contribution to \eqref{j-product} from non-coprime choices 
$u_1,\dots, u_t$ can be bounded as follows.
Suppose $(u_{j'},u_{j''})>1$. Then in particular
$(u_{j'},u_{j''})>N^{1/(\log \log N)^3}$, since any prime factor
of a $u_j$ is greater than $N^{1/(\log \log N)^3}$ by definition. Thus
$$\prod_{j\in [t]} \frac{2^{s_j}}{u_j} 
  \leq \frac{1}{N^{1/(\log \log N)^3}} \prod_{j\in[t]:j \not = j',j''}
  \frac{2^{s_j}}{u_j} 
\left(\frac{2^{2s_{j'}}}{u_{j'}} +\frac{2^{2s_{j''}}}{u_{j''}}
\right)~.$$
Since
$$\frac{W}{\phi(W)\gamma \log N} 
 \sum_{\substack{d_j \leq N^{\gamma}/v_j \\(d_j,u_jW)=1 }}\frac{1}{d_j}
\ll 1$$ the contribution to \eqref{j-product} from bad
$(u_i)_{i\in[t]}$ is at most
$$
\binom{t}{2} \frac{1}{N^{1/(\log \log N)^3}}
\prod_{j\in [t-1]} \sum_{s_j \geq 2/\gamma} 
 \sum_{\substack{i_j \geq \\ \log_2 s_j - 2}}
 \sum_{u_j \in U(i_j,s_j)} \frac{2^{2s_j}}{u_j}
\ll_t \frac{1}{N^{1/(\log \log N)^3}}~,
$$
where the convergence of the three nested sums follows from the proof of
Proposition \ref{majorant}.
This proves the claim.
\end{proof}

To summarise, we have shown that 
\begin{align*}
& \E_{n\in \Z^d \cap K} \prod_{j \in [t]}
   \nu'(\psi_j(n))
 = \E_{n\in \Z^d \cap K} \prod_{j \in [t]}
   \Bigg( 
   \sum_{s = 2/\gamma}^{(\log \log N)^3}
   \sum_{i = \log_2 s - 2}^{6 \log \log \log N}
   \sum_{u_j \in U(i,s)}
   1_{u_j | \psi_j(n)}
   {\tilde \tau}'_{\gamma}(\psi_j(n)) 
   \Bigg)
   \\
&= (1+o_{D}(1))
  \prod_{j \in [t]}
  \frac{W}{\phi(W) \gamma \log N}
  \Bigg( 
  \sum_{s = 2/\gamma}^{(\log \log N)^3} 
  \sum_{i =  \log_2 s - 2}^{6 \log \log \log N} 
  \sum_{u \in U(i,s)} \frac{2^s}{u}\sum_{v|u}
  \sum_{\substack{d \leq N^{\gamma}/v\\(d,uW)=1}} \frac{1}{d} 
  \Bigg) \\
&\quad + O\Big(\frac{N^{d-1 + O_D(\gamma)}}{\vol(K)}\Big) 
  + o_D(1)
  ~.
\end{align*}
Regarding the last equation in the special and already known case 
$\E_{n \leq N} \nu'(n) = 1 + o(1)$ of the linear forms condition implies
that each of the factors on the right hand side, which is independent of
$\Psi$, equals $C'(1+o(1))$. This completes the proof of Proposition
\ref{linear-forms-prop}.

\section{The correlation condition}

This section provides a proof of Proposition \ref{verification of
C-Condition}.

Due to the similar structure of our majorant to that of the majorant used
in \cite{green-tao-longprimeaps,green-tao-linearprimes}, the
function $\sigma_m$ can be chosen in the same manner as in
\cite{green-tao-longprimeaps,green-tao-linearprimes}.

\begin{proposition}[Green-Tao \cite{green-tao-longprimeaps}]
\label{sigma_m moments}
Let $\Delta: \Z \to \Z$ be the polynomial
defined by 
$\Delta(n) = \prod_{1 \leq j < j' \leq m} (\W n + b_{i_j} - b_{i_{j'}})$,
define $\sigma_m : \Z_{M'} \to \R^+$ to be 
$$ \sigma_m(n)
 := \exp \Bigg( \sum_{p>w(N),~p|\Delta(n)} O_m(p^{-1/2}) \Bigg) ~.$$
for $n>0$ and suppose $\sigma_m(0)=o(M')$. 
Then
$\E_{n \in \Z_{M'}} \sigma_m^q(n) \ll_{m,q} 1$.
\end{proposition}
\emph{Proof of Proposition \ref{verification of C-Condition}.}
The proof proceeds in two cases.
The first case considers the situation where $h_i=h_j$ for two
distinct indices $i,j$.
We aim to use the fact that on the right hand side of the
inequality 
$$\E_{n \in \Z_{M'}} \prod_{i \in [m]} \nu'_{\W,(b_i)}(n+h_i) 
  \leq \sum_{1 \leq i < j \leq m} \sigma_{m_0}(h_i - h_j)$$
$\sigma_{m_0}(0)$ occurs while Proposition \ref{sigma_m moments} allows us
to choose $\sigma_{m_0}(0)$ quite large. 
Indeed, H\"older's inequality yields
\begin{align*}
& \E_{n \in \Z_{M'}} \prod_{i\in [m]} \nu'_{\W,(b_i)}(n+h_i) \\
&= C'^{-m}\sum_{s_1, \dots , s_m} 
   \sum_{i_1,\dots, i_m}
   \sum_{\substack{u_1,\dots, u_m\\ u_j \in U(i_j,s_j), j\in[m]}}
   \E_{n \in \Z_{M'}}
   \prod_{\ell \in [m]} 
   {\tilde \tau}'_{\gamma} (\W (n+h_i)+b_{i_{\ell}})
   2^{s_{\ell}} 
   1_{u_{\ell}|(\W (n+h_i)+b_{i_{\ell}})} \\
&\leq C'^{-m}
   \sum_{s_1, \dots , s_m} 
   \sum_{i_1,\dots, i_m}
   \sum_{\substack{u_1,\dots, u_m\\ u_j \in U(i_j,s_j), j\in[m]}}
   \prod_{\ell \in [m]} 
   \Big(\E_{n \in \Z_{M'}}
   {\tilde\tau}'_{\gamma}{}^{m} (\W n+b_{i_{\ell}})
   2^{m s_{\ell}} 
   1_{u_{\ell}|(\W n+b_{i_{\ell}})} \Big)^{1/m}~.
\end{align*}
Since $\tau(n) \ll_{\eps} n^{\eps}$, we may continue this estimate by
\begin{align*}
&\ll_{\eps} \exp(\eps m \log N)
   \sum_{s_1, \dots , s_m} 
   \sum_{i_1,\dots, i_m}
   \sum_{\substack{u_1,\dots, u_m\\ u_j \in U(i_j,s_j), j\in[m]}}
   \prod_{\ell \in [m]} 
   \Big(\E_{n \in \Z_{M'}}
   2^{m s_{\ell}} 
   1_{u_{\ell}|(\W n+b_{i_{\ell}})} \Big)^{1/m}\\
&\ll \exp(\eps m \log N)
   \sum_{s_1, \dots , s_m} 
   \sum_{i_1,\dots, i_m}
   \sum_{\substack{u_1,\dots, u_m\\ u_j \in U(i_j,s_j), j\in[m]}}
   \prod_{\ell \in [m]} 
   \Big( 2^{m s_{\ell}} \frac{1}{u_{\ell}} \Big)^{1/m}\\
&\leq \exp(\eps m \log N)
   \prod_{\ell \in [m]}
   \sum_{s_{\ell}} 
   \sum_{i_{\ell}}
   \sum_{u_{\ell} \in U(i_{\ell},s_{\ell})}
   \frac{2^{s_{\ell}}}{u_{\ell}^{1/m}}~. 
\end{align*}
Note that the proof of Proposition \ref{majorant} implies that
$$\sum_{s > 2/\gamma}
 \sum_{i \geq \log_2 s - 2}
 \sum_{u \in U(i,s)}
 \frac{2^{s}}{u^{1/m}}
\leq \sum_{s \geq 2/\gamma}  \frac{2^s}{s^{s\gamma/(100m)}} 
  \left( 
  \sum_{j \geq 1} \left(\frac{100 \cdot e \cdot (\log 2 + o(1))}{\gamma j}
  \right)^{\gamma j/(100m)} \right)^s
$$ 
converges. Thus,
$$ 
 \E_{n \in \Z_{M'}} \prod_{i\in [m]} \nu'_{\W,(b_i)}(n+h_i)
 \ll_{m, \eps} \exp(\eps m \log N )~.
$$
Recall that $M'=N/\W=N^{1-o(1)}$, and therefore $N^{1/4}=o(M')$.
Choosing $\eps > 0 $ small enough so that $\eps m_0 < 1/4$ and setting
$$
\sigma_{m_0} (0) 
 := O_{m_0,\eps} \Big(\exp(\eps m_0 \log N) \Big)
 = O_{m_0,\eps}(N^{1/4})
 = o(M')~,
$$
we can ensure that
$$
 \E_{n \in \Z_{M'}} \prod_{i \in [m]} \nu'_{\W,(b_i)}(n+h_i) 
 \leq \sum_{1 \leq i < j \leq m} \sigma_{m_0}(h_i - h_j)
$$
when $h_i=h_j$ for some $i\not=j$.

\vspace{1\baselineskip}
\noindent
Next, we consider the case where $h_i \not= h_j$ whenever $i\not=j$.
Our approach to estimate 
$$\E_{n \leq M'} \prod_{j\in[m]} \nu'\left(\W(n+h_j)+b_{i_j}\right)$$
is the same as the one used to check the linear forms condition and we
therefore proceed to analyse the local divisor densities:
Since the forms $\psi_j(n)= \W(n+h_j)+b_{i_j}$ are affinely related, all
we can say in general for $p>w(N)$ is 
$$\alpha(p^{a_1}, \dots, p^{a_m}) = O(p^{-\max_i a_i})~.$$
If, however, more than one exponent $a_i$ is non-zero, then we have 
$$ \alpha(p^{a_1}, \dots, p^{a_m}) > 0 $$
only if $p ~\big|\left(\W(h_j-h_{j'}) + b_{i_j}-b_{i_{j'}}\right)$ for
some $j,j'\in [m]$.

\begin{Claim}
We have the following estimate
\begin{align}\label{bound for correlation estimate}
\E_{n \leq M'} \prod_{j\in[m]} \nu'\left(\W(n+h_j)+b_{i_j}\right)
&\ll \prod_{\substack{p|\Delta\\p>w(N)}} \sum_{a_1,\dots, a_m} 
 \alpha(p^{a_1},\dots,p^{a_m})
\end{align}
where
$\Delta := \prod_{j\not=j'}(\W(h_j-h_{j'})+b_{i_j}-b_{i_{j'}})$.
\end{Claim}

Before we prove the claim, we complete the verification of the
correlation estimate.
In order to apply the bound on $\alpha$, note that there are at most
$mj^{m-1}$ tuples $(a_1,\dots a_m)$ satisfying $\max_i a_i = j$.
For sufficiently large primes $p$, we have 
$$ mj^{m-1} < p^{j/2}/2~.$$
If furthermore $p>w(N)$ holds, as on the right hand side of 
\eqref{bound for correlation estimate}, then
$$\sum_{a_1,\dots, a_m} \alpha(p^{a_1},\dots,p^{a_m}) 
\leq 1 + \frac{1}{2}p^{-1/2}\sum_{j\geq 0} p^{-j/2} 
\leq 1 + p^{-1/2}$$
and therefore
\begin{align*}
 \prod_{\substack{p|\Delta \\ p>w(N)}}
 \sum_{a_1,\dots, a_m} \alpha(p^{a_1},\dots,p^{a_m})
 \ll \prod_{\substack{p>w(N)\\p|\Delta}}
 \left(1 + p^{-1/2}\right)~.
\end{align*}
Let 
$\Delta(n):= \prod_{j\not=j'}\big(\W n+b_{i_j}-b_{i_{j'}}\big)$ 
and set
$$
 \sigma_{m_0}(n)
 := \exp \Bigg( \sum_{p>w(N),~p|\Delta(n)} O_{m_0}(p^{-1/2}) \Bigg) ~.
$$
for $n>0$.
Since $1 + x \leq \exp x$, we have
\begin{align*}
\E_{n \leq M'} \prod_{j\in[m]} \nu'\left(\W(n+h_j)+b_{i_j}\right)
&\ll_m \sum_{1 \leq j < j' \leq m} 
 \sigma_{m_0} \left(h_j-h_{j'}\right)~.
\end{align*}
In view of the above Proposition \ref{sigma_m moments}, this completes the
verification of the correlation condition. 

\vspace{1\baselineskip}
\noindent
\emph{Proof of Claim.}
We have to bound the expression
\begin{align*}
& \E_{n \leq M'} \prod_{j\in[m]} \nu' \left(\W(n+h_j)+b_{i_j}\right) \\
&= C'^{-t}\E_{n \leq M'} \prod_{j \in [m]} 
 \sum_{s_j}
 \sum_{i_j}
 \sum_{u_j \in U(i_j,s_j)} 2^{s_j} {\tilde \tau}'_{\gamma}
 (\W(n+h_j)+b_{i_j}) 1_{u_j|(\W(n+h_j)+b_{i_j})}~.
\end{align*}
Dropping the normalising factor 
$\left(\frac{W}{C'\phi(W)\gamma \log N}\right)^t$
for the moment, the above becomes
\begin{align*}
&\E_{n \leq M'}
 \sum_{s_1,\dots,s_m}
 \sum_{i_1,\dots,i_m}
 \sum_{u_1,\dots,u_m} 
 \sum_{v_1|u_1, \dots, v_m|u_m}
 \sum_{\substack{d_1,\dots,d_m\\ d_i \leq N^{\gamma}/u_i\\(d_i,W)=1\\
 i=1,\dots,m}} \prod_{j\in [m]} 2^{s_j} 1_{u_jd_j|(\W(n+h_j)+b_{i_j})} \\
&\leq\E_{n \leq M'}
 \sum_{s_1,\dots,s_m}
 \sum_{i_1,\dots,i_m}
 \sum_{u_1,\dots,u_m}
 \sum_{\substack{d_1,\dots,d_m\\ d_i \leq N^{\gamma}/u_i\\(d_i,W)=1\\
 i=1,\dots,m}} \prod_{j\in [m]} 2^{s_j} 
 \tau(u_k)
 1_{u_jd_j|(\W(n+h_j)+b_{i_j})}~.
\end{align*}
This may be bounded as follows by the volume packing lemma, employed
together with the observations on
$\alpha(p^{a_1}, \dots, p^{a_m})$ we made just before the statement of
this claim.
\begin{align*}
&\ll 
 \sum_{s_1,\dots,s_m}
 \sum_{i_1,\dots,i_m}
 \sum_{u_1,\dots,u_m}
 \sum_{\substack{d_1,\dots, d_m \leq N^{\gamma}\\(d_j,W\Delta)=1}}
 \prod_{k=1}^m \left( 2^{s_k}\frac{\tau(u_k)}{u_k} \right) \frac{1}{d_k}
 \prod_{\substack{p|\Delta\\p>w(N)}} \sum_{a_1, \dots,a_m}
 \alpha(p^{a_1}, \dots, p^{a_m})\\
&\ll 
 \sum_{s_1,\dots,s_m}
 \sum_{i_1,\dots,i_m}
 \sum_{u_1,\dots ,u_m}
 \prod_{k=1}^m \left( 2^{s_k}\frac{\tau(u_k)}{u_k} \right)
 \prod_{\substack{w(N) < q <N^{\gamma}\\q \text{ prime}}}\! (1+ q^{-1})^t 
 \prod_{\substack{p|\Delta\\p>w(N)}} \sum_{a_1, \dots,a_m}
 \alpha(p^{a_1}, \dots, p^{a_m}).
\end{align*}
Noting that 
$$\left(\frac{W}{C'\phi(W)\gamma \log N}\right)^t 
  \prod_{\substack{w(N) < q < N^{\gamma}\\q \text{ prime}}} (1+q^{-1})^t
\ll 1~,$$
all that remains is to bound 
$$ \prod_{k=1}^m
 \sum_{s_k}
 \sum_{i_k}
 \sum_{u_k}
 2^{s_k}\frac{\tau(u_k)}{u_k}~.
$$
This, however, can be done in a similar way as in the proof of Proposition
\ref{majorant}:
\begin{align*}
 \sum_{s > 2/\gamma}
 \sum_{i \geq \log_2 s - 2}
 \sum_{u \in U(i,s)}
 2^{s} \frac{\tau(u)}{u}
&= \sum_{s > 2/\gamma}
 \sum_{i \geq \log_2 s - 2}
 \sum_{u \in U(i,s)}
 2^{s}2^{m_0(i,s)}\frac{1}{u} \\
&\leq \sum_{s > 2/\gamma}
 \sum_{i \geq \log_2 s - 2}
 2^{s}\frac{1}{m_0(i,s)!} (2\log 2 +o(1))^{m_0(i,s)} \\
&\leq \sum_{s> 2/\gamma}  \frac{2^{ms}}{s^{s\gamma/100}} 
  \sum_{j \geq 1} 
  \left(\frac{100 \cdot e \cdot (2\log 2 + o(1))}{\gamma j}
  \right)^{\gamma sj/100}\\
&\ll 1~,
\end{align*}
which completes the proof of the claim.

\section{Application of the transference principle}
\label{application_section}
The aim of this section is to deduce the main theorem from a generalised
von Neumann theorem and to prove some reductions on the remaining task of
checking that the conditions of the generalised von Neumann theorem are
satisfied.

The transference principle \cite[Prop. 10.3]{green-tao-linearprimes}
allows, as was discussed in Section \ref{W-trick-section}, to transfer
results that hold for bounded Gowers-uniform functions to Gowers-uniform
functions that are dominated by a pseudorandom measure.
It was developed in
\cite[\S8]{green-tao-longprimeaps} in view of an application to the
(unbounded) von Mangoldt function,
and was proved by an iteration argument.
New and simplified approaches to the transference principle were more
recently found by Gowers \cite{gowers} and
Reingold-Tulsiani-Trevisan-Vadhan \cite{rttv}.

The generalised von Neumann theorem asserts that, if $f$ is suitably
Gowers-uniform and dominated by a pseudorandom measure, then composing $f$
with linear forms $\psi_i$ that are sufficiently independent yields
functions $f \circ \psi_i$ that behave like independent variables: 
the expectation   
$\E_{n} \prod_{i\in [t]} f(\psi_i(n))$
is close to $(\E_n f(n))^t$, which were the expected value, 
had the $f \circ \psi_i$ genuinely been independent.

\begin{proposition}[Green-Tao \cite{green-tao-linearprimes}, generalised
von Neumann theorem]\label{ v.neumann}
Let $t,d,L$ be positive integer parameters.
Then there are constants $C_1$ and $D$, depending on $t,d$ and $L$, such
that the following is true.
Let $C$, $C_1 \leq C \leq O_{t,d,L}(1)$ be arbitrary and
suppose that $N' \in [CN,2CN]$ is a prime. 
Let $\nu: \Z_{N'} \to \R^{+}$ be a $D$-pseudorandom measure, and suppose
that $f_1,\dots,f_t : [N] \to \R$ are functions with 
$|f_i(x)| \leq \nu(x)$ for all $i \in [t]$ and $x \in [N]$. 
Suppose that $\Psi=(\psi_1,\dots,\psi_t)$
is a finite complexity system of affine-linear forms whose
linear coefficients are bounded by $L$.
Let $K\subset [-N,N]^d$ be a convex body such that 
$\Psi(K) \subset [1,N]^t$.
Suppose also that 
\begin{equation}\label{uniformity-condition;v.N.Thm}
 \min_{1\leq j \leq t} \|f_j\|_{U^{t-1}[N]}\leq \delta
\end{equation}
for some $\delta > 0$. 
Then we have
$$\sum_{n \in K} \prod_{i \in [t]} f_i(\psi_i(n)) 
  = o_{\delta}(N^d)~.$$ 
\end{proposition}

Establishing the Gowers-uniformity condition
\eqref{uniformity-condition;v.N.Thm} itself is a task that is
conceptually equivalent to that of finding an asymptotic for 
$\sum_{n \in K} \prod_{i\in [t]} f(\psi_i(n))$ directly, and should
therefore not be any easier.
The specific system of affine-linear forms that appears in the definition
of the uniformity norms, however, allows an alternative characterisation
of Gowers-uniform functions.

\stsubsection{A characterisation of Gowers-uniform functions}

Whether or not a function $f$ is Gowers-uniform is characterised by the
non-existence or existence of a polynomial nilsequence\footnote{For
definitions of nilmanifolds and nilsequences, see, for instance,
\cite{green-tao-polynomialorbits}.}
that correlates with $f$.
On the one hand, correlation with a nilsequence obstructs uniformity: 

\begin{proposition}[Green-Tao \cite{green-tao-linearprimes}, Cor. 11.6]
\label{nilsequences-obstruct-uniformity}
Let $s \geq 1$ be an integer and let $\delta \in (0,1)$ be real. 
Let $G/\Gamma = (G/\Gamma, d_{G/\Gamma})$ be an $s$-step nilmanifold with
some fixed smooth metric $d_{G/\Gamma}$ , and let $(F(g(n)\Gamma))_{n \in
\N}$ be a bounded $s$-step nilsequence with Lipschitz constant at most
$L$.
Let $f : [N] \to \R$ be a function that is bounded in the $L_1$-norm,
that is, assume $ \|f\|_{L_1} = \E_{n \in [N]} |f(n)| \leq 1$.
If furthermore
$$ \E_{n \in [N]} f(n) F(g(n)\Gamma) \geq \delta $$
then we have
$$ \|f\|_{U^{s+1}[N]} \gg_{s,\delta,L,G/\Gamma} 1~. $$
\end{proposition}

An inverse result to this statement has been known as Inverse Conjecture
for the Gowers norms ($\mathrm{GI}(s)$ conjectures) for some time and has
recently been resolved, see \cite{gtz}. 
The inverse conjectures are stated for bounded functions.
With our application to the normalised divisor function in mind, we only
recall the transferred statement,
c.f.~\cite[Prop. 10.1]{green-tao-linearprimes}, here:

\begin{proposition}[Green-Tao-Ziegler, 
Relative inverse theorem for the Gowers norms]
\label{inverse theorem} 
For any $0 < \delta \leq 1$ and any $C\geq 20$, there exists a finite
collection $\mathcal M_{s,\delta,C}$ of $s$-step nilmanifolds $G/\Gamma$,
each equipped with a metric $d_{G/\Gamma}$, such that the following
holds. 
Given any $N\geq1$, suppose that $N'\in[CN,2CN]$ is prime, that 
$\nu:[N'] \to \R^+$ is an $(s+2)2^{s+1}$-pseudorandom measure, suppose
that $f:[N] \to \R$ is any arithmetic function with $|f(n)|\leq \nu(n)$
for all $n\in [n]$ and such that $$\|f\|_{U^{s+1}[N]} \geq \delta~.$$
Then there is a nilmanifold $G/\Gamma \in M_{s,\delta,C}$ in the
collection and a $1$-bounded $s$-step nilsequence 
$(F(g(n)\Gamma))_{n \in \N}$ on it that has Lipschitz constant 
$O_{s,\delta,C}(1)$, such that we have the correlation estimate
$$|\E_{n\in [N]} f(n)F(g(n)\Gamma)|\gg_{s,\delta,C} 1~.$$
\end{proposition}

This inverse theorem now reduces the required uniformity-norm estimate
\eqref{uniformity-condition;v.N.Thm} to the potentially easier task of
proving that the centralised version of $f$ does not correlate with
polynomial nilsequences.

\stsubsection{Reduction of the main theorem to a non-correlation
estimate}
The task of proving the main theorem had been reduced to the proof of the
following proposition in Section \ref{W-trick-section}.

\begin{proposition_w-reduction2}
Let $M=N/\W$, let $\Psi:\Z^d \to \Z^t$ be a finite complexity system whose
linear coefficients are bounded by $L$.
Then for any choice of $b_1,\dots, b_t \in [\W]$ such that 
$\varpi(b_i)| (\W/W)$ for all $i \in [t]$,
\begin{equation*}
\E_{n \in \Z^d \cap K'}
\prod_{i \in [t]} {\tilde \tau}'(\W \psi_i(n) + b_i) 
= 1 + o_{d,t,L}(M^d/\vol(K'))
\end{equation*}
holds for every convex body $K' \subseteq [-M,M]^d$ which satisfies 
$\W \Psi(K')+(b_1,\dots,b_t)\subseteq[1,N]^d$.
\end{proposition_w-reduction2}
Define for $b \in [\W]$ the function $\tilde\tau'_{\W,b}: \Z \to \R$,
$\tilde\tau'_{\W,b}(n):= {\tilde \tau}'(\W n + b)$. Rewriting
\begin{equation*}
 \E_{n \in \Z^d \cap K'} 
 \prod_{i \in [t]} \tilde\tau'_{\W,b_i}(\psi_i(n)) - 1
 = \E_{n \in \Z^d \cap K'} \prod_{i \in [t]} 
   \Big( \left(\tilde\tau'_{\W,b_i}(\psi_i(n)) - 1 \right) + 1 \Big) - 1
\end{equation*}
and multiplying out the product on the right hand side, the constant term
cancels out, while all other terms are of a form the generalised von
Neumann theorem can be applied to, provided we can show that 
$$\| \tilde\tau'_{\W,b_i} - 1\|_{U^{t-1}} = o(1)$$
for all $i \in [t]$. 
By the inverse theorem, it thus suffices to establish the non-correlation
estimates
$$ |\E_{n\in [M]} (\tilde\tau'_{\W,b}(n) - \E \tau'_{w,b}) F(g(n)\Gamma)| 
 = o(1) $$
for all $(t-2)$-step nilsequences $F(g(n)\Gamma)$ as in 
Proposition \ref{inverse theorem} and all $b \in [\W]$ with
$\varpi(b)|\W$.

\section{Non-correlation of the $W$-tricked divisor function with
nilsequences}
\label{non-correlation-section}
The aim of this section is to provide the remaining non-correlation
estimate which will complete the proof of the main theorem.
For all concepts and notation in connection with nilmanifolds and
nilsequences that remain undefined in this section we refer to
\cite{green-tao-nilmobius} and its companion
paper \cite{green-tao-polynomialorbits}.

Let $k \geq 1$ be an arbitrary integer, let $F: G/\Gamma \to \C$ be a
Lipschitz function on the $(k-1)$-step nilmanifold $G/\Gamma$, and let 
$g: \Z \to G$ be a polynomial sequence adapted to some given
filtration $G_{\bullet}$ of $G$.

Let $b \in [\W]$ such that $\varpi(b)|\W$, and note that then
$\varpi(\W n+b)=\varpi(b)$.
The mean value of ${\tilde \tau}'_{\W,b}$ satisfies the following
identity
\begin{align*}
 \E_{n \leq M} {\tilde \tau}'_{\W,b}(n) 
= \frac{W}{\phi(W)\log N}  
  \sum_{\substack{d \leq N/\varpi(b)\\(W,d)=1 }} d^{-1} + o(1) 
= 1 + o(1)~.
\end{align*}
Indeed, employing the estimate
$(1-\log^{-1}N)^{\log \log N} 
= \exp(\frac{O(1)\log \log N}{\log N})
= 1 + o(1)$, we have in one direction
\begin{align*}
&\sum_{\substack{d \leq N/\varpi(b)\\(W,d)=1 }} d^{-1}
\prod_{p\leq w(N)} (1-p^{-1})^{-1} \\
&= (1+o(1))
   \sum_{\substack{d \leq N/\varpi(b)\\(W,d)=1 }} d^{-1}
   \prod_{p\leq w(N)} 
   \frac{1-p^{-\lfloor \log_p \log N \rfloor}}{1-p^{-1}} \\
&\leq (1+o(1)) \sum_{\substack{d \leq N(\log N)^{\log \log N}}}
  d^{-1} \\
&\leq (1 + (\log \log N)^2 / \log N)\log N + O(1) 
\leq (1+o(1))\log N~,
\end{align*}
and, in the other direction,
\begin{align*}
\sum_{\substack{d \leq N/\varpi(b)\\(W,d)=1 }} d^{-1}
\prod_{p\leq w(N)} (1-p^{-1})^{-1} 
\geq \sum_{\substack{d \leq N/\varpi(b)}} d^{-1}
= \log (N/\varpi(b)) + O(1)
=  (1 + o(1)) \log N.
\end{align*}
Setting 
\begin{align*}
\mu_{\W,b} 
 &:= \frac{W}{\phi(W)\log N}
    \sum_{\substack{d \leq (N/\varpi(b))^{1/2}\\(W,d)=1 }} 
    2(d^{-1} - \varpi(b)d/N)~, 
\end{align*}
we obtain
\begin{align*}
\mu_{\W,b} 
 &= \frac{2W}{\phi(W)\log N} \E_{n \leq N/\W} 
    \sum_{d:(d,W)=1} 1_{d|(\W n + b)}1_{d^2 < (\W n + b)/\varpi(b)}
    + o(1) \\
 &= \E_{n \leq M} {\tilde \tau}'_{\W,b}(n) = \mu_{\W,b} +o(1) = 1 + o(1)~.
\end{align*}
Hence, the application of the Gowers Inverse Theorem requires the
estimation 
\begin{align*}
&\E_{n \leq M} ({\tilde \tau}'_{\W,b}(n) - \mu_{\W,b}) F(g(n)\Gamma) \\
&= 2 \E_{n \leq M} 
  \sum_{\substack{d\leq (N/\varpi(b))^{1/2}\\ (d,W)=1}} 
  (1_{d|\W n+b}1_{\W n+b>d^2 \varpi(b)} - d^{-1}(1 - \varpi(b)d^2/N))
  F(g(n)\Gamma) \\
&= o_{F,G/\Gamma}(1)~.
\end{align*}

To achieve this, we shall employ the strategy and various lemmata from 
\cite{green-tao-nilmobius}. Some parts of the argument will be generalised
to meet our requirements.

The basic strategy is as follows.
When trying to establish a non-correlation estimate, it is desirable to
have good control on the nilsequence involved. 
This is for instance the case when the nilsequence is totally
equidistributed, that is, equidistributed in every sufficiently dense
subprogression of the range it is defined on.
While a nilsequence in general does not have this property, the
factorisation theorem from \cite{green-tao-polynomialorbits} states that
any nilsequence $g:[N] \to G$ may be written as a product 
$g(n)=\eps(n)g'(n)\gamma(n)$, where $\eps: [N] \to G$ is smooth, 
$g': [N] \to G'$ takes values in a rational subgroup $G' \leq G$ and
yields a totally equidistributed sequence on the corresponding
submanifold $G'/(G'\cap \Gamma)$ of $G/\Gamma$, and where $\gamma: [N] \to
G$ has the property that $n \mapsto \gamma(n)\Gamma$ is periodic.
 
The aim then is to show that, by passing to a collection of subsequences
defined on subprogressions of $[N]$, the correlation estimate involving
$g$ can be reduced to correlation estimates involving totally
equidistributed sequences arising from $g'$.

One further reduction is possible: Any periodic function of short period
can be regarded as a nilsequence.
Establishing non-correlation in the special case of periodic sequences is
likely to be much easier than the general case.
If we pass from $\E_{n \leq N} f(n)F(g(n)\Gamma)$ to considering the
collection $\frac{1}{N}\sum_{n\leq (N-i)/d} f(dn+i)F(g(dn+i)\Gamma)$ for 
$0\leq i < d$, where each sequence $g(dn+i)\Gamma$ takes values in some
subnilmanifold $G_i/\Gamma_i$ of $G/\Gamma$, then a non-correlation
estimate with periodic sequences allows us to assume that the mean values
$\int_{G_i/\Gamma_i}F(x)\, dx $ vanish.
Indeed, we may subtract off the periodic correlation
$$ \E_{n\leq N} f(n) 
 \left(\sum_{i=0}^{d-1} 1_{n\equiv i(d)} \int_{G_i/\Gamma_i} F(x)\right)
 = o(1)~,$$
that is, we may subtract off the relevant mean values.

This sketch shows the rough strategy from \S2 of
\cite{green-tao-nilmobius} for reducing a non-correlation
estimate to the case where the nilsequence is equidistributed and 
furthermore the involved Lipschitz function $F$ has zero mean.

The following is \cite[Thm. 1.1]{green-tao-nilmobius} adapted to our 
case. 

\begin{theorem}
Let $G/\Gamma$ be a nilmanifold of some dimension $m \geq 1$, let
$G_{\bullet}$ be a filtration of $G$ of some degree $d \geq 1$, and let
$g \in poly(\Z,G_{\bullet})$ be a polynomial sequence. Suppose that
$G/\Gamma$ has a $Q$-rational Mal'cev basis $\mathcal X$
for some $Q \geq 2$, defining a metric $d_{\mathcal X}$ on $G/\Gamma$.
Suppose that $F: G/\Gamma \to [-1,1]$ is a Lipschitz function. 
Recall that $M=N/\W$ and that the normalising factor of 
${\tilde \tau}'_{\W,b}$ depends on $N$. We have
$$|\E_{n \in [M]} {\tilde \tau}'_{\W,b}(n)F(g(n))\Gamma|
\ll_{m,d,\gamma,A} 
 Q^{O_{m,d,\gamma,A}(1)} 
 (1+ \|F\|)(\log\log\log N)^{-A}$$
for any $A>0$ and $N \geq 2$. 
\end{theorem}

\begin{proof}[Sketch proof:] 
Since $\E_{n\in[M]} |{\tilde \tau}'_{\W,b}(n)| = O(1)$, the theorem
trivially holds unless 
$Q \ll (\log \log \log N)^{O_{A,m,d}(1)} \ll w(N)$,
allowing us to assume
$Q^B \ll_B w(N)$, for some $B>1$ to be chosen later.

Proceeding as in \S 2 of \cite{green-tao-nilmobius}, one may reduce to
analysing the case where $\int F = 0$ and where $(g(n)\Gamma)$ is totally
$Q'^{-B}$-equidistributed for some $Q'$ with 
$Q \leq Q' \ll Q^{O_{B,m,d}(1)}$.
The necessary major arc estimate that allows us to assume $\int F = 0$ is
the following.
For any progression $P \subseteq [M]$ of common difference
$1 \leq q < w(N)$ we have
\begin{align*}
&\E_{n\in N} 1_P(n) ({\tilde \tau}'_{\W,b}(n) - \mu_{\W,b})\\
&= \frac{2W}{\phi(W)\log N}~ \E_{n \in N}~ 1_P(n)
   \sum_{\substack{d \leq (N/\varpi(b))^{1/2} \\ (d,W)=1}}
   \Big(1_{d|\W n+b}1_{\W n+b>d^2\varpi(b)} 
   - d^{-1}\Big(1 - \frac{d^2\varpi(b)}{N}\Big)\Big) \\
&\ll N^{-1/2}~.
\end{align*}
Note that this bound critically depends on the fact that all prime
divisors of $q$ are smaller than $w(N)$, which is ensured by the
assumption $1 \leq q < w(N)$.

The case where $(g(n)\Gamma)$ is totally $Q'^{-B}$-equidistributed and
$\int F = 0$ is a consequence of the next proposition
(cf.~also \cite[Proposition 2.1]{green-tao-nilmobius}), applied with
$\delta = Q'^{-B}$, provided $B$ was chosen large enough. 
\end{proof}

\begin{proposition}[${\tilde \tau}'_{\W,b}$ is orthogonal to
equidistributed nilsequences]\label{non-corr_tildetau'2}
Suppose that $G/\Gamma$ has a $Q$-rational Mal'cev basis $\mathcal X$
adapted to the filtration $G_{\bullet}$.
Suppose $g \in \mathrm{poly}(\Z,G_{\bullet})$ and that the finite
sequence $(g(n)\Gamma)_{n \leq M}$ is a totally
$\delta$-equidistributed in $G/\Gamma$. 
Then for any Lipschitz function $F:G/\Gamma \to [-1,1]$ with
$\int_{G/\Gamma} F =0$ and
for any progression $P\subset[M]$ of length at least $M/Q$, we have
$$|\E_{n\in[M]} ({\tilde \tau}'_{\W,b}(n) 1_{P}(n)F(g(n)\Gamma)| \ll
\delta^{c} Q^{O(1)} \|F\| \log \log \log N$$
for some $c=1/O_{d,m}(1)$.
\end{proposition}

For the proof of this proposition we employ tools from the analysis of
\emph{Type I} sums in the proof of
\cite[Proposition 2.1]{green-tao-nilmobius}.
The main ingredient is the following lemma which generalises the
aforementioned \emph{Type I} sums analysis.
Since large parts of the highly technical proof remain virtually
unchanged, we chose to only outline the argument to that extent which
enables us to describe the parts new to it.
In the first part of the proof, we follow the presentation of 
\cite[\S 3]{green-tao-nilmobius} closely.

\begin{lemma}\label{typeIsums}
Suppose that $(g(n)\Gamma)_{n \leq M}$ is a
totally $\delta$-equidistributed sequence, suppose that
$F:G/\Gamma \to [-1,1]$ is a Lipschitz function with Lipschitz constant
$\|F\|_{\mathrm{Lip}}=1$.
Suppose further that $\delta > N^{-\sigma}$ for some $\sigma \in (0,1)$,
and that $Q \leq \delta^{-c_1}$ for some parameter $c_1 \in (0,1)$.
Let $P \subseteq [M]$ be a progression of length at least $M/Q$.
Then, provided that $\sigma$ and $c_1$ are sufficiently small, depending
only on the degree of $g$ and the dimension of the nilmanifold $G/\Gamma$,
the following holds.
For any $1 \leq K \leq N^{1/2}$ there are only $o(\delta^{O(c_1)})K$
values of $k$ satisfying $k \in (K,2K]$ and
$$ \left|k^{-1} 
  \E_{N/\W < n < 2N/\W} ~
  1_{k|\W n + b} 
  1_P(\W n + b)
  F(g(n)\Gamma)\right| 
\gg \delta^{O(c_1)}~.$$
\end{lemma}

\begin{proof}
 Suppose for contradiction that there is some $K$, 
$1 \leq K \leq N^{1/2}$, such that the following inequality holds for $\gg
\delta^{O(c_1)} K$ values of $k \in (K,2K]$
\begin{align*}
&\left|\frac{1}{k} 
  \E_{N/\W < n < 2N/\W} ~
  1_{k|\W n + b} 
  1_P(\W n + b)
  F(g(n)\Gamma)\right| \\
& = 
 \left|\E_{N/\W k < m < 2N/\W k} ~
  1_P(\W(km + u_k) + b) F(g(k m + u_k)\Gamma) \right| \\
& \gg \delta^{O(c_1)}~,
\end{align*}
where $u_k$ is the smallest integer for which $k| \W u_k + b$. 
$u_k$ exists for all $k$ for which the inequality holds.
To remove the indicator function of $P$, let $\ell \leq Q$ denote the
common difference of $P$ and split the range of $m$ into progressions of
common difference $\ell$.
Pigeonholing shows that there is some residue $b' \Mod{\ell}$ such that we
still find $\gg \delta^{O(c_1)}K$ values of $k \in (K,2K]$ that satisfy
\begin{align}\label{k-condition}
\left| \sum_{m' \in I_k } F(g(k (\ell m' + b') + u_k)) \right| \gg
\delta^{O(c_1)} \frac{N}{\W k \ell}~,
\end{align}
where $I_k \subseteq [N/\W2k\ell -1, N/\W k\ell ]$ is an interval.
This lower bound means that for those $k$ that satisfy
\eqref{k-condition}, the sequence $\tilde{g}_k : \Z \to G$ defined by
$\tilde{g}_k(n):=g(k (\ell n + b') + u_k)$,
fails to be $\delta^{O(c_1)}$-equidistributed in $G/\Gamma$ on the range
$N_k = [N/\W2k\ell - 1, N/\W k\ell]$.

By \cite[Thm 2.9]{green-tao-polynomialorbits} there is a
non-trivial horizontal character $\psi_k: G \to \R/\Z$ of modulus
$|\psi_k| \ll \delta^{O(c_1)}$ such that
$$\| \psi_k \circ \tilde{g}_k \|_{C^{\infty}[N_k]} \ll \delta^{-O(c_1)}.$$
For notational simplicity, we remove the dependence on $b$ and $\ell$.
This step is not strictly necessary for the proof.
Let $g_k : \Z \to G$ be defined by $g_k(n):=g(k n + u_k)$.
Then \cite[Lemma 8.4]{green-tao-nilmobius} asserts that there is some
integer $q_k$, $1 < q_k \ll \delta^{-O(c_1)}$ such that
$$\| q_k \psi_k \circ g_k \|_{C^{\infty}[N_k]} \ll \delta^{-O(c_1)}~.$$
Pigeonholing over the possible choices of horizontal character
$q_k \psi_k$, there is some non-trivial $\psi$ of modulus $|\psi| \ll
\delta^{O(c_1)}$ among them such that
$$\|\psi \circ g_k \|_{C^{\infty}[N_k]} \ll \delta^{-O(c_1)}$$
for $\gg \delta^{O(c_1)}K$ values of $k \in (K, 2K]$.
Let $$\psi \circ g (n) = \beta_d n^{d} + \dots + \beta_0 $$
be the projection of the polynomial sequence to $\R/\Z$ by the character
$\psi$.
Then 
$$
 \psi \circ g_k(n) = \beta_d k^d n^d + (\text{lower order terms in }n)~.
$$
We now consider just the highest coefficients $\beta_d k^d$.
As in \cite[p.9]{green-tao-nilmobius}, one shows that $\beta_d$ is close
to a rational with small denominator, more precisely, that there is some
$\tilde q, 1 \leq \tilde q \ll \delta^{-O(c_1)}$
\begin{align} \label{beta_d-close-to-rational}
\|\tilde q \beta_d \|_{\R/\Z} \ll \delta^{-O(O(c_1))} (N/\W)^{-d}~.
\end{align}
Behind this is the following: since $\psi \circ g_k$ has small smoothness
norm, the coefficients, in particular $\beta_d k^d$, are close to
rationals with small denominator. 
Waring's theorem states that one can express many integers as a sum of few
$d$th powers. 
This allows us to show that $\beta_d n$ is strongly recurrent in $\R/\Z$,
and hence $\beta_d$ is close to a rational with small denominator.
[This part requires that $\sigma$ is sufficiently small.]

Up to here the argument has followed \cite[\S3]{green-tao-nilmobius}; 
the new changes come in now.
The bound \eqref{beta_d-close-to-rational} means that $\beta_d n^{d}$
varies very slowly on progressions of common difference $\tilde q$.
By pigeonholing, one of these progressions, say 
$\{n \equiv q' \mod \tilde q \}$, contains the numbers $u_k$ for at least
$\gg \delta^{O(c_1)} K$ of our selection of values $k \in (K,2K]$
that also satisfy \eqref{k-condition}.

For each such $k$ consider the full expansion of $\psi \circ g_k$:
\begin{align*}
 \psi \circ g_k(n) 
&= \sum_{j=1}^d \beta_j(k n+ u_k)^j\\
&= \beta_d k^d n^d  
 + \Big( \beta_{d-1} + \binom{d}{1} u_k \beta_{d} \Big) 
   k^{d-1} n^{d-1} \\
&\qquad + \left( \beta_{d-2} + \binom{d-1}{1} u_k \beta_{d-1} + 
     \binom{d}{2} u_k^2 \beta_{d}
   \right)
   k^{d-2} n^{d-2} + \dots ~. 
\end{align*}
Since $u_k \equiv q' \Mod{\tilde q}$, there are integers 
$a_{d-1}, \dots, a_0$ such that 
$$\left\| \binom{d}{j} u_k^{j} \beta_{d} 
  - \frac{a_j}{{\tilde q}} \right\|_{\R/\Z} 
\ll \delta^{-O(c_1)} (N/\W)^{-d}~.$$
We aim to use this information to remove the appearance of the $u_k$,
which vary with $k$ in a way we have no control on, from the
coefficient of $n^{d-1}$, hoping to then run a similar argument as
before to show that $\beta_{d-1}$ is close to being rational.

Writing 
$$\psi \circ g_k(n) = \sum_{j=1}^d \tilde \beta_{j,k} k^j n^j~,$$
the assertion 
$$\| q \tilde \beta_{j,k} k^j \|_{\R/\Z} 
\ll (N/\W K)^{-j} \|\psi \circ g_k\|_{C^{\infty}[N_k]} 
\ll (N/\W K)^{-j} \delta^{-O(c_1)}$$
holds if and only if
$$\left\| q \Big(\tilde \beta_{j,k} - \binom{d}{j} u_k^{j} \beta_{d} 
  + \frac{a_j}{{\tilde q}} \Big)k^j \right\|_{\R/\Z} 
\ll (N/\W K)^{-j} \|\psi \circ g_k\|_{C^{\infty}[N_k]}
\ll (N/\W K)^{-j} \delta^{-O(c_1)}~.$$
Thus, we can remove all occurrence of $\beta_d$ in the $\tilde \beta_j$
for $j < d$.
For $j=d-1$ this also removes all occurrences of $u_k$, since
$$\tilde \beta_{d-1,k}=\beta_{d-1} + \binom{d}{1} u_k \beta_{d}~.$$ 
We proceed inductively:
We know that there is $q=O(1)$ such that for $\gg \delta^{O(c_1)}K$
values of $k$ from our selection of $k \in (K,2K]$ the following holds
$$
\| q k^{d-1} (\beta_{d-1} + \frac{a_{d-1}}{\tilde q}) \|_{\R/\Z} 
\ll (N/\W K)^{-d+1} \delta^{-O(c_1)}~.$$
As before, one deduces via Waring's theorem that 
$\beta_{d-1} + \frac{a_{d-1}}{\tilde q}$, and
hence $\beta_{d-1}$ is close to a rational with small denominator, say
$\tilde{\tilde{q}}$.
Pass to a subprogression of common difference $\tilde{\tilde{q}}$ such
that for many of our $k$ the number $u_k$ belongs to that subprogression,
note that we can remove the appearance of $\beta_{d-1}$ in all $\tilde
\beta_j$ for $j< d-1$, and the appearance of $u_k$ in 
$\tilde \beta_{d-2}$.
Show that $\beta_{d-2}$ is close to a rational with small denominator and
repeat.

Finally, we see that there is $\bar q, 1\leq \bar q \ll \delta^{-O(c_1)}$ 
such that 
$$\| \bar q \beta_j \|_{\R / \Z} \ll \delta^{-O(c_1)} N^{-j}~.$$

This means that $\|\bar q \psi \circ g (n) \|_{\R/\Z}$ is small on a
reasonably long interval: exactly as in \cite{green-tao-nilmobius},
we have for fixed small $\eps>0$, e.g. $\eps = 1/10$,
$$\|\bar q \psi \circ g (n) \|_{\R/\Z} 
\ll n \delta^{-O(c_1)}N^{-1}
\leq \eps$$
for all $n\leq N'= \delta^{Cc_1}N$ provided $C$ is large enough.

Consider the Lipschitz function $\tilde F: G/\Gamma \to [-1,1]$ that
arises as composition of $\bar q \psi$ with a smooth cut-off of the
interval $[-\eps,\eps]$, where the cut-off has Lipschitz constant $O(1)$.
Since 
$\|\bar q \psi\|_{\mathrm{Lip}} \ll |\bar q \psi| \leq \delta^{-O(c_1)}$,
we have $\|F\|_{\mathrm{Lip}} \ll \delta^{-O(c_1)}$.
Thus, if $c_1$ is sufficiently small then 
$$|\E_{n \in [N'] } \tilde F (g(n) \Gamma) | 
\geq 1 > \delta \| \tilde F \|_{\text{Lip}}~,$$
which contradicts the assumption that $g$ was $\delta$-equidistributed
and hence proves the lemma.
\end{proof}

\begin{proof}[Proof of Proposition \ref{non-corr_tildetau'2}] 
Following the reductions from the start of the proof of 
\cite[Proposition 2.1]{green-tao-nilmobius},
one shows that the result is trivially true in all cases that are not
covered by the assumptions of Lemma \ref{typeIsums}.

Since $(g(n)\Gamma)_{n \leq M}$ is totally $\delta$-equidistributed and
since $\int_{G/\Gamma} F = 0$, it suffices to show that 
$$|\E_{n\in[M]} {\tilde \tau}'_{\W,b}(n) 1_{P}(n)F(g(n)\Gamma)| \ll
\delta^{O(1)} \log \log \log N ~.$$
This, however, follows from Lemma \ref{typeIsums} via dyadic summation:
 \allowdisplaybreaks 
\begin{align*}
 &\E_{n\in[M]} {\tilde \tau}'_{\W,b}(n) 1_{P}(n)F(g(n)\Gamma)\\
 &= \frac{2W}{\phi(W)\log N} 
    \sum_{\substack{d < (N/\varpi(b))^{1/2}\\(d,W)=1}} M^{-1}~
    \sum_{d^2 \varpi(b)/\W \leq n \leq M} 
    1_{d|\W n+b} 1_{P}(n)F(g(n)\Gamma) \\
 &\leq \frac{2W}{\phi(W)\log N} 
     \sum_{\substack{j \leq  \\ \frac{1}{2} \log_2(N/\varpi(b))}}
     \sum_{\substack{d \sim 2^j \\(d,W)=1}}
     \sum_{\substack{\ell : d^2 \leq 2^{\ell} \leq M}}
     \frac{2^{\ell-1}}{M}
|\E_{n \in [2^{\ell-1},2^{\ell}]}
     1_{d|\W n+b} 1_{P}(n)F(g(n)\Gamma)| \\
 &\ll \frac{W}{\phi(W)\log N} 
     \sum_{\substack{j \leq  \\ \frac{1}{2} \log_2(N/\varpi(b))}}
     \Bigg( \sum_{\substack{d \sim 2^j \\(d,W)=1}}
     \sum_{\substack{\ell : d^2 \leq 2^{\ell} \leq M}}
     \frac{2^{\ell-1} d}{M}
     \delta^{O(1)} 
     + \delta^{O(1)} 2^{j} 
     \sum_{\substack{\ell : 2^{2j} \leq 2^{\ell} \leq M2^{-j}}}
     \frac{2^{\ell-1} 2^{j}}{M}\Bigg)\\
&\ll \frac{W}{\phi(W)\log N} 
     \Bigg( \sum_{\substack{d<N^{1/2}\\ (d,W)=1}}
       d^{-1} \delta^{O(1)}
     + \delta^{O(1)} \log_2 N 
     \Bigg) \\
 &\ll \delta^{O(1)} \log w(N) \\
 &\ll \delta^{O(1)} \log\log\log N~.
\end{align*} 
\end{proof}

\subsection*{Acknowledgements}

I should like to thank my PhD supervisor Ben Green for suggesting the
original problem that led to the results of this paper, Ben Green and
Tom Sanders for many valuable discussions and encouragement, and Tim
Browning for drawing my attention to the paper \cite{nair-tenenbaum}.
I am particularly grateful to the members of two reading groups for all
their comments, corrections and suggestions. 
Namely, R{\'e}gis de la Bret{\`e}che's group in Paris
(Arnaud Chadozeau, Sary Drappeau and Pierre Le Boudec) and Tim
Browning's group in Bristol (Thomas Bloom, Julia Brandes, Julio Cesar,
Eugen Keil, Siu Lun Alan Lee, Gihan Marasingha, Damaris Schindler and
Michael Swarbrick Jones).

\providecommand{\bysame}{\leavevmode\hbox to3em{\hrulefill}\thinspace}


\begin{thebibliography}{99}


\bibitem{birch} B.~J.~Birch,
\emph{Multiplicative functions with non-decreasing normal order,}
{J. London Math. Soc.} 
{\bf 42} (1967), no.~1, 149--151.

\bibitem{browning}
T.~D.~Browning
\emph{The divisor problem for binary cubic forms,}
{J. Th{\'e}orie Nombres Bordeaux,} to appear.
Preprint available at
{arXiv:1006.3476v1}.

\bibitem{erdos} P.~Erd\H{o}s,
\emph{On the sum $\sum_{k=1}^x d(f(k))$,} 
{J. London Math. Soc.} {\bf 27} (1952), no.~1, {7--15}.

\bibitem{gowers} W.~T.~Gowers,
\emph{Decompositions, approximate structure, transference, and the
Hahn-Banach theorem,}
{arXiv:0811.3103v1},
preprint.

\bibitem{green-tao-longprimeaps} B.~J.~Green and T.~C.~Tao,
\emph{The primes contain arbitrarily long arithmetic progressions,}
{Annals of Math.} {\bf 167 } (2008), No.~2, 481--547.

\bibitem{green-tao-linearprimes} \bysame, % B.~J.~Green and T.~C.~Tao, 
\emph{Linear equations in primes,} 
{Annals of Math.,} {\bf 171 } (2010), No.~3 , 1753--1850.

\bibitem{green-tao-nilmobius} \bysame, % B.~J.~Green and T.~C.~Tao, 
\emph{The M\"obius function is strongly orthogonal to nilsequences,}
{arXiv:0807.1736v3},
preprint.

\bibitem{green-tao-polynomialorbits} \bysame, % B.~J.~Green and T.~C.~Tao,
\emph{The quantitative behaviour of polynomial orbits on nilmanifolds,}
{Annals of Math.,} to appear.
Preprint available at 
{arXiv:0709.3562v5}.

\bibitem{gtz} B.~J.~Green, T.~C.~Tao and T.~Ziegler
\emph{An inverse theorem for the Gowers $U^k[N]$ norm},
{arXiv:1009.3998v2}, 
preprint.

\bibitem{hardy-wright} G.~H.~Hardy, E.~M.~Wright,
\emph{An Introduction to the Theory of Numbers,}
{Oxford University Press}, 
{Fifth Edition}, 
{1979}.

\bibitem{Ing27JLMS} A.~E.~Ingham. 
\emph{Some asymptotic formulae in the theory of numbers,} 
{J. London Math. Soc.}, {\bf 2} (1927), no.~3, 202--208.

\bibitem{m-diagquadratics} L.~Matthiesen,
\emph{Linear correlations amongst numbers represented by positive definite
binary quadratic forms,} 
{arXiv:1106.4690v1},
preprint.


\bibitem{nair-tenenbaum} {M.~Nair, G.~Tenenbaum},
\emph{Short sums of certain arithmetic functions,}
{Acta Math.} {\bf 180} {(1998)}, {no.~1}, {119--144}.

\bibitem{rttv} {O.~Reingold, L.~Trevisan, M.~Tulsiani and S.~Vadhan},
\emph{Dense Subsets of Pseudorandom Sets,}
{In: Proceedings of the 49th IEEE Symposium on Foundations of
Computer Science} (2008), 76--85.

\bibitem{tenenbaum} G.~Tenenbaum,
\emph{Lois de r\'epartition des diviseurs. {IV},}
{Ann. Inst. Fourier (Grenoble)} {\bf 29 } (1979), no.~3, 1--15.

\end{thebibliography}
\end{document}